\documentclass[11pt]{amsart}
\usepackage[utf8]{inputenc}
\usepackage{amsmath}
\usepackage{amssymb}
\usepackage{xcolor}
\usepackage[normalem]{ulem}
\usepackage{caption}
\usepackage{tikz}
\tikzstyle{vertex}=[circle, fill, draw, inner sep=0pt, minimum size=6pt]
\newcommand{\vertex}{\node[vertex]}

\usepackage{amsthm}
\usepackage{centernot}
\usepackage[top=1in, bottom=1in, left=1in, right=1in]{geometry}

\usepackage{enumitem, hyperref}
\makeatletter
\def\namedlabel#1#2{\begingroup
    #2%
    \def\@currentlabel{#2}%
    \phantomsection\label{#1}\endgroup
}

\newtheorem{theorem}{Theorem}[section]
\newtheorem{prop}[theorem]{Proposition}
\newtheorem{proposition}[theorem]{Proposition}
\newtheorem{corollary}[theorem]{Corollary}
\newtheorem{lemma}[theorem]{Lemma}
\newtheorem{conjecture}[theorem]{Conjecture}

\theoremstyle{definition}
\newtheorem{example}[theorem]{Example}
\newtheorem{definition}[theorem]{Definition}

\newtheorem{procedure}[theorem]{Procedure}
\newtheorem{remark}[theorem]{Remark}

\newtheorem{assumption}[theorem]{Assumption}

\newcommand{\bbR}{\mathbb{R}}

\newcommand{\calC}{{\mathcal{C}}}
\newcommand{\calD}{{\mathcal{D}}}

\newcommand{\calU}{{\mathcal{U}}}
\newcommand{\calV}{{\mathcal{V}}}
\newcommand{\calW}{{\mathcal{W}}}

\newcommand{\Lk}[2]{\text{Lk}_{#1}(#2)}
\newcommand{\MaxInt}[1]{\text{Max}_{\cap}(#1)}
\newcommand{\Man}[1]{\text{Man}(#1)}

\newcommand{\code}{\calC}

\newcommand{\realiz}{\calU}

\newcommand{\fullreal}{\realiz = \{U_i\}_{i=1}^n}

\newcommand{\scplex}{\Delta}

\newcommand{\scplexC}{\Delta(\code)}
\newcommand{\scplexCstar}{\Delta(\code^{\star})}
\newcommand{\scplexD}{\Delta(\calD)}

\newcommand{\scplexCres}{\Delta(\code)|_{\sigma \cup \tau}}

\newcommand{\scfaceS}{\sigma}

\newcommand{\scfaceT}{\tau}

\newcommand{\LkDelta}{\Lk{\scfaceS}{\scplex}}

\newcommand{\LkCfaceS}{\Lk{\scfaceS}{\scplexC}}

\newcommand{\LkCfaceT}{\Lk{\scfaceT}{\scplexC}}

\newcommand{\LkCres}{\Lk{\scfaceS}{\scplexCres}}

\newcommand{\MaxIntC}{\MaxInt{\scplexC}}

\newcommand{\ManC}{\Man{\scplexC}}

\newcommand{\mincode}{\code_{\rm min}(\scplex)}
\newcommand{\mincodestar}{\code_{\rm min}(\Delta(\code^*))}

\newcommand{\spokeOne}{\scfaceS_1}

\newcommand{\spokeTwo}{\scfaceS_2}

\newcommand{\spokeThree}{\scfaceS_3}

\newcommand{\spokei}{\scfaceS_i}

\newcommand{\spokej}{\scfaceS_j}

\newcommand{\spokek}{\scfaceS_k}

\newcommand{\spokeell}{\scfaceS_{\ell}}

\newcommand{\hub}{(\scfaceS_1 \cup \scfaceS_2 \cup \scfaceS_3)}

\newcommand{\rim}{\scfaceT}

\newcommand{\Wheel}{\calW}

\newcommand{\WheelFull}{\calW = (\spokeOne, \spokeTwo, \spokeThree, \rim)}

\newcommand{\constrOne}{\rho_1}

\newcommand{\constrThree}{\rho_3}

\newcommand{\constrj}{\rho_j}

\newcommand{\constrBad}{\rho}

\newcommand{\treelink}{\text{TL}}

\newcommand{\locobOne}{\scfaceS}
\newcommand{\locobTwo}{\scfaceT}

\newcommand{\tk}{\mathrm{Tk}}

\begin{document}

\title{
Wheels: A New Criterion for non-convexity of Neural Codes}
\author{Alexander Ruys de Perez, Laura Felicia Matusevich, and Anne Shiu}
\address{Department of Mathematics\\Texas A\&M University\\College Station, TX 77843, USA}
\date{\today}

\begin{abstract}
   We introduce new geometric and combinatorial criteria that preclude a neural code from being
   convex, and use them to tackle the classification problem for codes
   on six neurons.
   Along the way, we give the first example of a code that is non-convex, has no local obstructions, and has simplicial complex of dimension two.
   We  also characterize convexity for neural codes for which the simplicial complex is pure of low or high dimension.
   
\vspace{0.1 in}
\noindent
\end{abstract}
\maketitle

\section{Introduction} \label{sec:intro}

A {\em neural code} on $n$ neurons is a subset $\code \subseteq 2^{[n]}$,
where $[n]=\{1,2,\dots,n\}$.  Of interest are neural codes that describe the regions of a configuration of convex open sets $U_1,U_2, \dots, U_n$ in Euclidean space.  Such {\em convex codes} arise from neurons in the hippocampus, specifically, {\em place cells}.  In this case, each region $U_i$ is called a {\em place field} and represents the subset of an animal's environment  where the corresponding neuron fires.  Convex codes, therefore, allow the brain to map out an environment.  Place cells were found by John O'Keefe in 1971~\cite{Oke1}, for which he was awarded a joint 
(with May-Britt Moser and Edvard Moser)
Nobel Prize in Physiology or Medicine in 2014.

The central open question in the mathematical theory of neural codes is to determine which codes are convex.  
Some families of codes are known to be convex~\cite{cruz2019open, what-makes, sparse, williams}, 
and convex codes on up to 5 neurons are well characterized~\cite{what-makes, curto2013neural, goldrup2020classification}.  
As for precluding convexity, the main tool is a combinatorial criterion known as a local obstruction~\cite{what-makes,no-go} and its generalizations~\cite{new-obs, Jeffs2018convex}.  

Possessing a local obstruction is sufficient, but not necessary, for a code to be non-convex~\cite{lienkaemper2017obstructions}.  
Accordingly, this article introduces a new tool for ruling out convexity.
We show that this obstruction -- which we call a wheel -- 
captures all non-convex codes (with no local obstructions) on up to 5 neurons (Theorem~\ref{thm:5-neurons-complete}).
Additionally, we give combinatorial criteria for possessing a wheel, which we use to classify many codes on 6 neurons.  

In our classification, a crucial step is to eliminate codes that are essentially equivalent to codes on fewer neurons.  These include {\em decomposable} codes for which, at least intuitively, a realization is obtained by placing a realization of a smaller {\em embedded code} within an {\em ambient code}. The key result is that a decomposable code on up to $6$ neurons is convex if and only if its embedded and ambient codes are convex (Theorem~\ref{thm:decomposable}).  The proof relies on prior
classifications of codes on up to 5 neurons~\cite{what-makes, curto2013neural, goldrup2020classification} 
as well as results on {\em nondegenerate convexity} due to Cruz {\em et al.}~\cite{cruz2019open}.


We also prove results pertaining to a code and the dimension of its simplicial complex.  
We present a code that is non-convex, has no local obstructions, and has simplicial complex of dimension two (Proposition~\ref{prop:3-sparse-counterex}).  This example answers a question posed by Chen, Frick, and Shiu~\cite{new-obs}.
Finally, we characterize convexity for neural codes for which the simplicial complex is pure of low or high dimension (Theorem~\ref{thm:pure}).

The outline of this article is as follows.
We begin with background material on neural codes in Section~\ref{sec:preliminaries}. 
In Section~\ref{sec:wheel}, we introduce wheels, and then give 
combinatorial criteria for them in Section~\ref{sec:search}. Next, 
Sections~\ref{sec:search-2} and~\ref{sec:reduce-decompose} -- which the reader might wish to skip on a first reading -- contain results that allow us to search computationally for wheels, and
explain how codes can be decomposed into smaller codes.  These results are then used in Section~\ref{sec:6-neuron} to obtain a partial classification of codes on 6 neurons. Section~\ref{sec:pure} presents our results on codes with simplicial complexes that are pure, and we end with a discussion in Section~\ref{sec:discussion}.

\section{Preliminaries}
\label{sec:preliminaries}

In this section we introduce basic notions and review necessary background.

A \textit{code on $n$ neurons} is a subset $\code \subseteq 2^{[n]}$,
where $[n]=\{1,2,\dots,n\}$. Elements of a code $\code$ are called
\textit{codewords}, and a \textit{maximal codeword} is a codeword in $\code$ that is maximal with respect to inclusion. 

\begin{assumption} \label{assum:empty-codeword}
All codes in this article are assumed to contain the empty
codeword $\varnothing$.  Whether or not a code contains $\varnothing$ does not affect convexity, defined below~\cite[Remark 2.19]{new-obs}.
\end{assumption}

A \textit{realization} of a code $\code$ in $\bbR^d$ is a collection
$\fullreal$ of open subsets of 
a \textit{stimulus space} $X \subseteq \mathbb{R}^d$
such that $c \in \code$ if and
only if 
 $(\cap_{i\in c}U_i) \smallsetminus (\cup_{j \in [n] \smallsetminus c}U_j) \neq \varnothing$.
The set $(\cap_{i\in c}U_i) \smallsetminus (\cup_{j \in [n] \smallsetminus c}U_j)$ is the 
\textit{atom of $c$ in the realization $\realiz$.} 
We denote  $U_{\scfaceS}:= \cap_{i\in \scfaceS}U_i$, 
and, by convention,  $U_{\varnothing}:=X$. 
If the sets $U_1,U_2, \dots,U_n$ are convex, then $\realiz$ is a
\textit{convex realization} of $\code$. Codes that possess convex
realizations are known as \textit{convex codes}.  As we are assuming that all codes contain $\varnothing$, we can always take $X=\mathbb{R}^d$ as the stimulus space; cf.~\cite[Remark 2.19]{new-obs}.


In what follows, we write a codeword without the braces, e.g., ``123'' rather than 
``$\{1,2,3\}$''.  We also indicate maximal codewords by boldface.

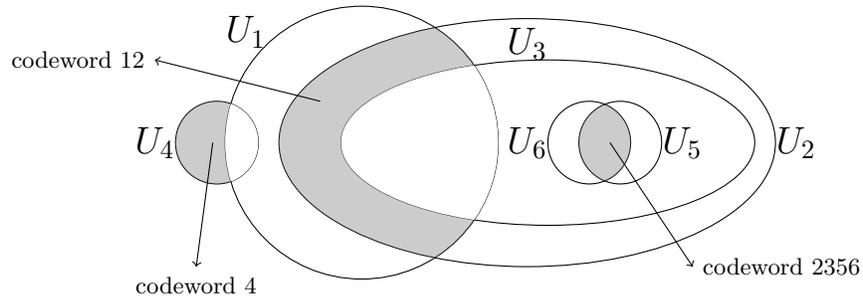
\begin{figure}[ht]
\begin{center}
\begin{tikzpicture}[scale=0.55]

\def\UOne{(0,0) circle (3.3cm)}
\def\UTwo {(4,0) ellipse (6cm and 3cm)}
\def\UThree {(4.5,0) ellipse (5cm and 2 cm)}
\def\UFour {(-3.5,0) circle (1cm)}
\def\UFive {(5.5,0) circle (1cm)}
\def\USix {(6.25,0) circle (1cm)}

\tikzset{invclip/.style={clip,insert path={{[reset cm]
        (-16383.99999pt,-16383.99999pt) rectangle (16383.99999pt,16383.99999pt)}}}}

\draw\UOne; 
\draw\UTwo;
\draw\UThree;
\draw\UFour;
\draw\UFive;
\draw\USix;
\draw (-2.8,2.687) node {\Large$U_1$};
\draw (10.5,0) node {\Large$U_2$};
\draw (4,2.4) node {\Large$U_3$};
\draw (-5,0) node {\Large$U_4$};
\draw (7.75,0) node {\Large$U_5$};
\draw (4,0) node {\Large$U_6$};

\begin{scope}[fill opacity=.4] 
\clip \USix;
\fill[gray] \UFive; 
\end{scope}  

\begin{scope}[fill opacity = .4]
\fill[gray] \UFour;
\end{scope}

\begin{scope}[fill opacity = 1]
\clip \UOne;
\fill[white]\UFour;
\end{scope}

\begin{scope}[fill opacity = .4]
\clip \UTwo;
\fill[gray] \UOne;
\end{scope}

\begin{scope}[fill opacity = 1]
\clip \UThree;
\fill[white] \UOne;
\end{scope}

\draw[->](-3.6,0) -- (-4,-3) node[anchor=north] {\footnotesize codeword 4} ;

\draw[->](-1,1) -- (-5,2) node[anchor=east] {\footnotesize codeword 12} ; 

\draw[->](6,0) -- (8,-3) node[anchor=west] {\footnotesize codeword 2356} ; 
\end{tikzpicture}

\caption{
Displayed is a realization of the code in Example~\ref{ex:code-decomposable}; each set $U_i$ is the open region inside 
the labeled ellipse. The atoms of three codewords are shaded.} \label{fig:real-6neur}
\end{center}
\end{figure}

\begin{example} \label{ex:code-decomposable}
The code 
$\code = \{ {\bf 2356, 123, 14}, 235, 236, 12, 23, 1,2,4,\varnothing\}$ on $6$ neurons is convex, as shown by the convex realization depicted in Figure~\ref{fig:real-6neur}.
\end{example}

Given a code $\code$ on $n$ neurons and a subset $\scfaceS \subseteq [n]$, the \textit{trunk} of $\scfaceS$ in $\code$, denoted by $\tk_{\code}(\scfaceS)$, is the set of all codewords containing $\scfaceS$:
    \[
    \tk_{\code}(\scfaceS):=\{c\in \code\ |\ \scfaceS \subseteq c\}~.
    \]
Next, every code $\code$ on $n$ neurons gives rise to a simplicial complex on
$[n]$, called the \textit{neural complex} of $\code$ (or the \textit{simplicial complex of $\code$}), and denoted by $\scplexC$, as follows:
\[
  \scplexC := \{ \scfaceS \in 2^{[n]} \mid \scfaceS \subseteq c \text{ for
    some } c \in \code \}.
\]
Note that the facets of $\scplexC$ are the maximal codewords of
$\code$. 
If $\scfaceS$ is a face of a simplicial complex~$\scplex$,
we write $\dim(\scfaceS):= |\scfaceS |-1$, and
$\dim(\scplex):= \max
\{\dim (\scfaceS)\ |\ \scfaceS \in \scplex\}$, where $\dim$ stands for
\textit{dimension}. 

We remark that if $\realiz$ is any realization of a code $\code$, then
$\scplexC$ is the \textit{nerve} of $\realiz$.
By definition, this means that, for $\scfaceS \in 2^{[n]}$, 
\[
  \scfaceS \in \scplexC \quad \text{ if and only if } \quad \bigcap_{i\in
    \scfaceS} U_i\neq \varnothing.
\]
As $\scplexC$ is simultaneously the nerve of all realizations of
$\code$, it captures important topological and geometric
characteristics of $\code$. 

It has been a fruitful approach to explore convexity of neural codes
based on a study of the associated neural complexes. The key notion is
that of a \textit{local obstruction}, which is given below. Codes with local obstructions
are known to be \textit{not} convex
(Proposition~\ref{prop:summary-prior-results}).
Local obstructions have been generalized recently~\cite{new-obs, Jeffs2018convex}, but for simplicity we do not present these generalizations here.

\begin{definition}
  \label{def:localObstruction}
Let $\code$ be a code on $n$ neurons, and assume that $\code$ 
is realized by a collection $\fullreal$ of open subsets of some $\mathbb{R}^d$.  A \textit{local obstruction} of $\code$ is a pair $(\locobOne, \locobTwo)$
of nonempty, disjoint subsets of $[n]$ such that 
 $$U_{\locobOne} ~\subseteq ~ \bigcup_{j\in \locobTwo}U_j$$
and the link $\LkCres$ is not contractible.
\end{definition}

Links are usually written as ${\rm Lk}_{\scplex}(\scfaceS)$, rather than $\LkDelta$; however, we follow the notation of~\cite{what-makes}.

Checking for local obstructions directly from the definition implies
analyzing not just $\scplexC$ but its restriction to every subset of
the vertices. However, obstructions can
always be detected at the level of $\scplexC$, a phenomenon we call
\textit{bubbling up} (see Proposition~\ref{prop:summary-prior-results}(\ref{itm:second}) below).  
We first 
introduce some necessary terminology.

\begin{definition} Let $\scplex$ be a simplicial complex on $[n]$, and let
  $\scfaceS \in \scplex$. 
  \begin{enumerate}[label=(\roman{*}), ref=(\roman{*})]
    \item $\scfaceS$ is a \textit{max-intersection face} of
      $\scplex$ if it is the intersection of two or more facets of $\scplex$.
      \item $\scfaceS$ is a \textit{mandatory face}
        of $\scplex$ if
        $\scfaceS \neq \varnothing$ and 
        the link $\LkDelta$ is not contractible.
      \end{enumerate}
Let $\code$ be a code on $n$ neurons.
  \begin{enumerate}[label=(\roman{*}), ref=(\roman{*})]    
      \item A \textit{max-intersection face} of  $\code$ is a
        max-intersection face of $\scplexC$. We denote by $\MaxIntC$
        the collection of all such faces.  The code $\code$
        is \textit{max-intersection-complete} if  $\MaxIntC \subseteq \code$.
       \item  A \textit{mandatory face} of $\code$ is a mandatory face
         of $\scplexC$. We denote by $\ManC$ the collection of all such
         faces.
\end{enumerate}
\end{definition}

The following result summarizes several prior results on convexity (see \cite[Theorem 1.2]{cruz2019open}, \cite[Lemma~1.4, Theorem 1.3, Proposition 1.7, Supplementary-Information Lemma 6.2]{what-makes}, and \cite[Theorem 5]{no-go}). 
\begin{prop} \label{prop:summary-prior-results}
 Let $\code$ be a code on $n$ neurons.
 \begin{enumerate}
     \item \label{itm:first}
     Every mandatory face of $\code$ is the intersection of two or more maximal codewords of $\mathcal{C}$ (that is, $\ManC \subseteq \MaxIntC$).
     \item \label{itm:second}
     $\code$ is max-intersection-complete (i.e., $\MaxIntC \subseteq \code$) $\Rightarrow$ $\code$ is convex $\Rightarrow$ $\code$ has no local obstructions $\Leftrightarrow$ $\code$ contains all mandatory faces of $\code$ (i.e., $\ManC \subseteq \code$).
     \item \label{itm:third}
    If $n \leq 4$, then both implications in (\ref{itm:second}) are equivalences.  
 \end{enumerate}
 \end{prop}

It follows that the \textit{minimal code} of a simplicial complex $\scplex$, defined as follows:
\begin{align} \label{eq:min-code}
 \mincode
 ~:=~
 \{\textrm{facets of $\scplex$}\} \cup \{\textrm{mandatory faces of $\scplex$} \} \cup \{\varnothing \}~,
\end{align}
is the unique minimal (with respect to inclusion) code among all codes with neural complex $\scplex$ and no local obstructions.

But what of those codes that contain the mandatory faces but not all
the max-intersection faces? The next examples show that such
codes may or may not be convex.  Indeed, we see that two implications in Proposition~\ref{prop:summary-prior-results}\eqref{itm:second} 
are, in general, \textit{not} equivalences.

\begin{example}[A code that is convex but not max-intersection-complete] \label{ex:convex-notmaxint}
The following neural code 
is not max-intersection-complete 
($1=123 \cap 134 \cap 145$ is missing):
\[
\code = \{
{\bf 123}, {\bf 134}, {\bf 145}, 13, 14, 
\varnothing\}~.
\]
However, $\code$ is convex, 
as shown in~\cite[Figure 3c]{what-makes}.
\end{example}

\begin{example}[A non-convex code with no local obstructions]
\label{ex:lienshiu}
The 
following neural code was the first example of a 
non-convex code with no local obstructions~\cite[Theorem 3.1]{lienkaemper2017obstructions}:
\[
\code^{\star} = \{ {\bf 2345}, {\bf 123}, {\bf 134}, {\bf 145}, 13, 14, 23, 34, 45, 3, 4, \varnothing \}~.
\]
To see that $\code^{\star}$ indeed has no local obstructions, first observe 
that the intersections  of the maximal codewords are
$23$, $34$, $45$, $13$, $14$, $1$, $3$, $4$, and $\varnothing$, each of which -- except $1$ -- is a codeword of $\code^{\star}$. However, $1$ is not mandatory, 
as the neural complex has facets $2345, ~ 123, ~ 134,$ and $145$,
and so the link of $1$ in $\scplexCstar$ is the following path graph, which is contractible:
\begin{tikzpicture}[scale=0.6]
        \vertex[label=$2$](p1) at (0,0) {};
        \vertex[label=$3$](p2) at (1,0) {};
        \vertex[label=$4$](p3) at (2,0) {};
        \vertex[label=$5$](p4) at (3,0) {};
        \path [-] (p1) edge node[above] {} (p2);
        \path [-](p2) edge node[above] {} (p3);
        \path [-](p3) edge node[above] {} (p4);
\end{tikzpicture}
\end{example}

\section{Wheels} \label{sec:wheel}

The following definition abstracts the type of non-convexity appearing
in the code in Example~\ref{ex:lienshiu}. 

\begin{definition} \label{def:wheel}
Let $\code$ be a code on $n$ neurons, and let $\fullreal$ be a
realization of $\code$. A tuple
$\Wheel=(\spokeOne,\spokeTwo,\spokeThree,\rim)\in (2^{[n]})^4$ is a
\textit{wheel of the realization $\realiz$} if it satisfies:
\begin{description}
    \item[\namedlabel{itm:Wi}{W(i)}] $U_{\spokej}\cap U_{\spokek} =  U_{\spokeOne}\cap U_{\spokeTwo}\cap U_{\spokeThree} \neq \varnothing$ for all $1\leq j < k \leq 3$, 
    \item[\namedlabel{itm:Wii}{W(ii)}] $U_{\spokeOne}\cap
      U_{\spokeTwo}\cap U_{\spokeThree}\cap U_{\rim} = \varnothing$,
    \item[\namedlabel{itm:Wiv}{W(iii)}]
    if $U_{\rim}$ and 
    $ U_{\spokej} \cap U_{\rim} $ 
    are convex for $j=1,2,3$,
    then 
    there exists a line segment whose endpoints lie one in
    $U_{\spokeOne}\cap U_{\rim}$
    and the other in $U_{\spokeThree}\cap U_{\rim}$, that also meets
    $U_{\spokeTwo}\cap U_{\rim}$.
 \end{description}

The sets $\spokeOne$, $\spokeTwo$, and $\spokeThree$ are the
\textit{spokes} 
and $\rim$ is the \textit{rim}. 
Also, $\Wheel$ is a \textit{wheel of the code $\code$} if it is a
wheel of \uline{every} realization of $\code$. 

We remark that~\ref{itm:Wiv} is related to but does not imply 
(see Remark~\ref{rmk:extra-condition-not-in-def}) 
the following condition:  
  \begin{description}
     \item[\namedlabel{itm:Wiii}{W(iii)\textsubscript{$\circ$}}]
       $U_{\spokej}\cap U_{\rim} \neq \varnothing$ for $j = 1,2,3$.
    \end{description}

For brevity, we say that $\realiz$ (or $\code $) {\em has a wheel} if there exists a wheel of $\realiz$ (respectively, $\code$).
\end{definition}

\begin{remark}[Condition~\ref{itm:Wiii}] 
\label{rmk:extra-condition-not-in-def}
Some wheels do not satisfy~\ref{itm:Wiii} (see Example~\ref{ex:wheel-frame-trivial} below).
However, requiring this condition -- which, unlike condition~\ref{itm:Wiv}, is easy to check -- 
often makes it easier to work with wheels (see the next section).
\end{remark}

\begin{remark}  \label{rmk:role-3-spokes}
In Definition~\ref{def:wheel}, the spoke regions $U_{\spokeOne}$ and
$U_{\spokeThree}$ seem to play a different role than 
$U_{\spokeTwo}$, but this is not really the case. A more symmetric way
of stating~\ref{itm:Wiv} would be to ask for a line segment that intersects
all three sets $U_{\spokej}\cap U_{\rim}$, for $j=1,2,3$. We have adopted
the current numbering convention to simplify the writing in some of
the proofs below.
\end{remark}

In the next section, we show that 
every realization of  
the code 
$\code ^{\star}$ from Example~\ref{ex:lienshiu}
has a wheel (see Proposition~\ref{prop:original-counterexample-wheel-frame}).

Wheels are relevant because they forbid convexity:

\begin{theorem} \label{thm:wheel}
Let 
$\realiz$ be a realization of a neural code $\code$. If
$\realiz$ has a wheel, then $\realiz$ is not a convex
realization. Consequently, if $\code$ has a wheel, then 
$\code$ is non-convex. 
\end{theorem}

\begin{proof}
Let $(\spokeOne,\spokeTwo,\spokeThree,\rim)$ be a wheel of the realization
$\fullreal$ of $\code$ in $\bbR^d$, and assume for contradiction that $\realiz$ is
convex. Let $L$ be the line segment from~\ref{itm:Wiv}.  Since
$U_{\rim}$ is convex, $L\subseteq U_{\rim}$.  

The sets $U_{\spokei}$, for $i=1,2,3$,  are convex and
open. Condition~\ref{itm:Wi} implies that $L$ intersects
$U_{\spokeOne}\cap U_{\spokeTwo}\cap U_{\spokeThree}$ by \cite[Lemma 3.2]{lienkaemper2017obstructions}. As $L\subseteq 
U_{\rim}$, it follows that $U_{\spokeOne}\cap U_{\spokeTwo} \cap
U_{\spokeThree}\cap U_{\rim}\neq \varnothing$, which contradicts~\ref{itm:Wii}.
\end{proof}

The geometric intuition behind the proof of Theorem~\ref{thm:wheel}
(and the reason for our chosen terminology) is
that in a realization that has a wheel
$\WheelFull$, the $U_{\spokei}$'s force
$U_{\rim}$ to be non-convex by bending $U_{\rim}$ around their
intersection $U_{\scfaceS_1} \cap U_{\scfaceS_2} \cap U_{\scfaceS_3}$. 
  See Figure~\ref{fig:wheel}.

\begin{figure}[ht]
\begin{center}
\resizebox{0.5\textwidth}{0.25\textwidth}{
\begin{tikzpicture}
\draw[blue] (0,0) to (0,2) to (2,2) to (2,0) to (0,0);
\draw[blue] (0,0) to (-9,0) to (-9,2) to (0,2) to (0,0);
\draw[blue] (2,0) to (2,2) to (11,2) to (11,0) to (2,0);
\draw[blue] (0,2) to (0,11) to (2,11) to (2,2) to (0,2);
\draw[red] (0,10) arc (90:180:8);
\draw[red] (10,2) arc (0:90:8);
\draw[red] (0,8) arc (90:180:6);
\draw[red] (8,2) arc (0:90:6);
\draw[red] (10,2) to (10,-1) to (8,-1) to (8, 2);
\draw[red] (-8,2) to (-8,-1) to (-6,-1) to (-6, 2);
\draw[red] (2,10) to (0,10);
\draw[red] (2,8) to (0,8);
\node[black] at (-4,1) {\Huge $U_{\scfaceS_1}$};
\node[black] at (6,1) {\Huge $U_{\scfaceS_3}$};
\node[black] at (1,6) {\Huge $U_{\scfaceS_2}$};
\node[red] at (6,8) {\Huge $U_{\scfaceT}$};
\draw [fill=gray, opacity=0.2] (0,0) rectangle (11,2);
\draw [fill=blue, opacity=0.2] (-9,0) rectangle (2,2);
\draw [fill=magenta, opacity=0.2] (0,0) rectangle (2,11);
\end{tikzpicture}
}
\caption{
Depicted is a conceptual picture of a wheel, with intuition as follows: If $U_{\scfaceS_1}$, $U_{\scfaceS_2}$, and $U_{\scfaceS_3}$ are convex, then 
the region of the rim, $U_{\scfaceT}$, ``bends around'' the intersection of the spoke regions, $U_{\scfaceS_1} \cap U_{\scfaceS_2} \cap U_{\scfaceS_3}$, and so is non-convex.} \label{fig:wheel}

\end{center}
\end{figure}
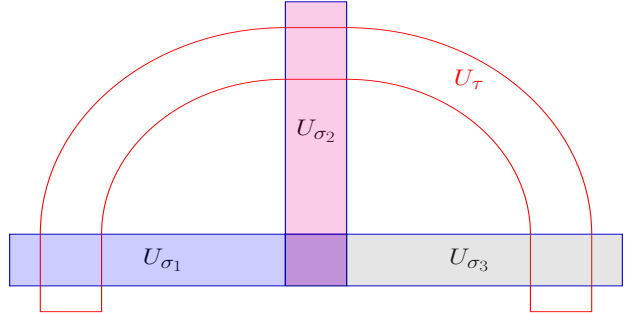

\begin{remark}[Open versus closed sets] \label{rmk:open}
Theorem \ref{thm:wheel} shows that wheels prevent codes from being realized by convex, open sets.  However, some codes with wheels can be realized by convex, \uline{closed} sets. 
Such a realization is shown in~\cite{cruz2019open} for the code in Example~\ref{ex:lienshiu}.  
\end{remark}

\begin{remark}[Relation to sunflowers] \label{rmk:sunflower} 
Wheels are closely related to ideas in recent work of Jeffs~\cite{jeffs2019sunflowers}.  As a start, 
in a realization $\realiz$ that has a wheel (or at least satisfies~\ref{itm:Wi}), the sets $U_{\scfaceS_1}$, $U_{\scfaceS_2}$, and $U_{\scfaceS_3}$ form what Jeffs calls a {\em 3-petal sunflower}.  
Jeffs also uses sunflowers to construct an infinite family of non-convex codes $\code_2, \code_3, \dots$~\cite[Definition 4.1]{jeffs2019sunflowers}.  
We show in Example~\ref{ex:jeffsc2} that $\code_2$ contains a wheel.  The codes 
$\code_m$, for $m\geq 3$, 
have sunflowers with $m+1$ petals.
The concept of a wheel
has not yet been generalized to such higher-dimensional cases, where,
instead of a line intersecting the 3 petals of a sunflower, 
a $(d-1)$-dimensional affine space 
intersects $d+1$ petals (cf.\ \cite[Theorem 1.1]{jeffs2019sunflowers}). Indeed, we checked that 
the codes 
$\code_m$, for small $m\geq 3$, do not have the combinatorial wheels introduced in the next section.  Moreover, we expect that all $\code_m$, for $m\geq 3$, lack wheels; however, checking condition~\ref{itm:Wiv} is difficult.
Finally, we note that, unlike prior work on using sunflowers to preclude convexity, our work provides algorithms (see Procedure \ref{procedure}) for doing so.
\end{remark}

\begin{example}\label{ex:jeffsc2}
The following neural code was introduced by Jeffs~\cite[Definition 4.1]{jeffs2019sunflowers}:
$$ \code_2 ~=~ 
    \{ {\bf 1236}, 
    {\bf 234}, {\bf 135}, {\bf 456}, 
    13, 23, 4, 5, 6,
    \varnothing    \}~.$$ 
This code is non-convex~\cite[Theorem 4.2]{jeffs2019sunflowers}
 and, 
 among all non-convex codes, 
  is minimal with respect to a certain partial order  (relabeling the neurons via the permutation $(1,4,2,6,3,5)$ yields the code $\code_0$ in~\cite[Theorem 5.10]{jeffs2020morphisms}).
We will see 
in Example~\ref{ex:jeffsc2-again}
that $\code_2$ 
contains a wheel. 
\end{example}

Finally, we highlight that wheels, together with local obstructions, completely characterize non-convexity in codes on up to 5 neurons (cf.~\cite[Theorem~3.1]{goldrup2020classification}):

\begin{theorem} \label{thm:5-neurons-complete}
A code $\code$ on up to 5 neurons is convex if and only if $\code$ has no local obstructions and 
no wheel frames. 
\end{theorem}

A ``wheel frame'', which is defined in the next section, 
is a combinatorial object 
whose presence implies the existence of wheels 
(Theorem~\ref{prop:wire-frame}). 
We prove Theorem~\ref{thm:5-neurons-complete} at the end of Section~\ref{sec:reduce-decompose}.

\section{Combinatorics of wheels} \label{sec:search}
We introduced wheels using realizations (Definition~\ref{def:wheel}),
because they are convenient for proving non-convexity. That being said,
the true goal of this article is to obtain non-convexity criteria in
terms of $\code$ and $\scplexC$, in a similar way to how 
the mandatory faces of $\scplexC$ 
determine whether $\code$ has local obstructions (recall Proposition~\ref{prop:summary-prior-results}\eqref{itm:second}).

It turns out that conditions~\ref{itm:Wi},~\ref{itm:Wii}, and
\ref{itm:Wiii} from Definition~\ref{def:wheel} can be restated
combinatorially, and depend only on the code $\code$ and not on the
specific realization $\realiz$ (Proposition~\ref{prop:D} below). 
The geometric condition~\ref{itm:Wiv}, however, 
is more subtle. Indeed, it is an open problem to recast~\ref{itm:Wiv} completely in
combinatorial terms, or show that no such characterization exists.
Nevertheless, we are able to provide combinatorial
criteria that imply the existence of wheels in every realization of
$\code$ (see Section~\ref{sec:combinatorial-wheels}).

\subsection{Combinatorial versions of \ref{itm:Wi}, \ref{itm:Wii}, and \ref{itm:Wiii}}

We start with a useful technical result.

\begin{lemma}\label{lem:ABC}
Consider a code $\code$ on $n$ neurons,
a realization $\fullreal$ of
$\code$, 
and subsets $\varphi,\psi_1,\psi_2,\dots,\psi_r 
\subseteq [n]$.
Then $U_{\varphi}\subseteq \bigcup_{t=1}^r U_{\psi_t}$ if
and only if 
$\tk_{\code}(\varphi)\subseteq \bigcup_{t=1}^r \tk_{\code}(\psi_t)$.
\end{lemma}

\begin{proof} 

$(\Rightarrow)$ Suppose  $U_{\varphi}\subseteq \bigcup_{t=1}^r
U_{\psi_t}$. 
Let $ c\in \tk_{\code}(\varphi)$, that is, $\varphi \subseteq c \in \code$. (We must show that $ \psi_s  \subseteq c$ for some $1 \leq s \leq r$.)  
The containment $\varphi \subseteq c$ yields the first containment below:
\begin{align} \label{eq:containment-1-lemma}
U_c 
    ~\subseteq~
U_\varphi
    ~\subseteq~ \bigcup_{i=1}^r U_{\psi_t}~,
\end{align}
and the second containment is by hypothesis.
Next, $c \in \code$ 
implies that $U_c\smallsetminus (\bigcup_{i\not\in
  c}U_i ) \neq \varnothing$. 
Combining this inequality with~\eqref{eq:containment-1-lemma} yields $U_{\psi_s} \smallsetminus (\bigcup_{i\not\in
  c}U_i) \neq \varnothing$ for some $s$, which implies $\psi_s
\subseteq c$. 
%
%
%
%
%
%

$(\Leftarrow)$ Assume that
$\tk_{\code}(\varphi)\subseteq \bigcup_{t=1}^r \tk_{\code}(\psi_t)$, 
that is, $\varphi \subseteq c \in \code$ implies that $\psi_s \subseteq c$ for some $s$.
Let $p\in U_{\varphi}$;
we must show that $p \in U_{\psi_s}$ for some $s$.
Define $c = \{i\in [n]\ |\
p\in U_i\}$. 
By construction, 
$p\in
U_c\smallsetminus \bigcup_{i\not\in c}U_i$;
so, $c \in \code$.  
Also by construction (of $p$ and $c$),
$\varphi \subseteq c$.  
Hence, by hypothesis, there exists $s$ such that $\psi_s
\subseteq c$. 
So, $p \in U_i$ for all $i \in \psi_s$. In other words, $p \in U_{\psi_s}$. 
\end{proof}

We are now ready to combinatorially recast part of Definition~\ref{def:wheel}.

\begin{definition}
  \label{def:partialWheel}
  A tuple $\WheelFull\in (\scplexC)^4$ is a \textit{partial-wheel} of
  a neural code $\code$ if it satisfies the following conditions:
\begin{description}
    \item[\namedlabel{itm:Di}{P(i)}] $\spokeOne \cup \spokeTwo \cup
      \spokeThree \in \scplexC$, and 
      $\tk_{\code}(\spokej \cup \spokek ) = \tk_{\code}(\spokeOne \cup \spokeTwo \cup \spokeThree)$ for every $1\leq j < k \leq 3$, 
    \item[\namedlabel{itm:Dii}{P(ii)}]$\spokeOne \cup \spokeTwo \cup
      \spokeThree \cup \rim \not\in \scplexC$, and 
    \item[\namedlabel{itm:Diii}{P(iii)\textsubscript{$\circ$}}]
      $\spokej \cup \rim \in \scplexC$ for $j = 1,2,3$. 
    \end{description}
\end{definition}

\begin{prop}[Equivalence of \ref{itm:Wi} and \ref{itm:Di}, etc.] \label{prop:D}
Let $\code$ be a neural code, and let $\WheelFull\in (\scplexC)^4$. Using the notation from Definition~\ref{def:partialWheel}, the following are equivalent:
\begin{itemize}
    \item[(1)] $\calW$ satisfies \ref{itm:Wi} (respectively \ref{itm:Wii}, \ref{itm:Wiii}) for some realization $\realiz$ of $\code$,
    \item[(2)] $\calW$ satisfies \ref{itm:Wi} (respectively \ref{itm:Wii}, \ref{itm:Wiii}) for all realizations $\realiz$ of $\code$, and
    \item[(3)] $\calW$ satisfies \ref{itm:Di} (respectively \ref{itm:Dii}, \ref{itm:Diii}). 
\end{itemize}
\end{prop}

\begin{proof}
The fact that $\scplexC$ is the nerve of every realization 
of $\code$
directly implies the equivalences involving \ref{itm:Wii} and \ref{itm:Wiii}. For the equivalences involving \ref{itm:Wi}, use Lemma \ref{lem:ABC} with 
$r = 1$, 
$\varphi = (\spokei \cup \spokej)$, and $\psi = \hub$. 
\end{proof}

\subsection{Combinatorial wheels} \label{sec:combinatorial-wheels}
Our next goal is to give (combinatorial) criteria in terms of $\code$ and~$\scplexC$
that imply the existence of wheels (and therefore non-convexity of
$\code$). We introduce three such criteria, and we call the resulting
wheel a sprocket, wire wheel, or wheel frame (see
Propositions~\ref{prop:sprocket},~\ref{prop:wire-wheel}, and Theorem~\ref{prop:wire-frame}).

\begin{definition}
  \label{def:sprocket}
A \textit{sprocket} of a neural code $\code$ is a partial-wheel (Definition~\ref{def:partialWheel})
$\WheelFull$ of $\code$ that in addition satisfies:
  \begin{description}
    \item[\namedlabel{itm:Div-rho}{S(iii)}] 
    There exist $\constrOne$, $\constrThree \in \scplexC$
    such that: 
        \begin{description}
            \item[\namedlabel{itm:Div-rho-1}{S(iii)(1)}] 
            $\tk_{\code}(\spokej \cup \rim)\subseteq \tk_{\code}(\constrj)$ for $j = 1,3$,
            \item[\namedlabel{itm:Div-rho-2}{S(iii)(2)}] 
            $\tk_{\code}(\rim)\subseteq \tk_{\code}(\constrOne) \cup \tk_{\code}(\constrThree)$, and 
            \item[\namedlabel{itm:Div-rho-3}{S(iii)(3)}] 
            $\tk_{\code} (\constrOne \cup \constrThree \cup \rim)\subseteq \tk_{\code}(\spokeTwo)$.
        \end{description}
\end{description}
 The sets $\constrOne$ and $\constrThree$ are {\em witnesses} for the sprocket.
  \end{definition}


To explain the terminology, a sprocket is a toothed wheel, such as the
gear wheel on a bicycle (see Figure~\ref{fig:sprocket}); we imagine the sets $U_{\constrOne}$ and
$U_{\constrThree}$ as overlapping links in a roller chain set on a sprocket.

\begin{figure}
  \includegraphics[width=1in]{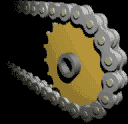}
  \caption{
  A sprocket.  This image is a still of an animation of the point of contact between a chain and a sprocket, which was produced by the US Department of Labor and is in the public domain~\cite{sprocket-image}. Conceptually, the witnesses $\constrOne$ and $\constrThree$ fill the role of aligning the spokes $\spokeOne$, $\spokeTwo$, and $\spokeThree$, similar to how the links of the chain keep the physical sprocket in place.
    \label{fig:sprocket}}
\end{figure}



The following result shows that sprockets are wheels.
  
\begin{prop}[Every sprocket is a wheel]\label{prop:sprocket}

Let $\code$ be a code on $n$ neurons, and 
assume that 
$\WheelFull\in (2^{[n]})^4$ 
satisfies \ref{itm:Wiii} (or equivalently \ref{itm:Diii}). 
If $\Wheel$ satisfies condition \ref{itm:Div-rho} from
Definition~\ref{def:sprocket}, then $\Wheel$ satisfies
\ref{itm:Wiv}. In particular, if $\Wheel$ is a sprocket of $\code$,
then $\Wheel$ is a wheel of $\code$. Consequently, 
 codes with sprockets are non-convex.
\end{prop} 
 
\begin{proof}

We must show that \ref{itm:Div-rho}
and \ref{itm:Diii} 
together 
imply \ref{itm:Wiv}. Using Lemma~\ref{lem:ABC}, we see
that~\ref{itm:Div-rho} is equivalent to the following condition in
terms of realizations (``G'' here stands for ``geometric''): 
\begin{description}
\item[\namedlabel{itm:Giv-rho}{G(iii)}] 
There exist $\constrOne, \constrThree \subseteq [n] $ 
    such that in every
realization $\fullreal$ of $\code$,
    \begin{description}
        \item[\namedlabel{itm:Giv-rho-1}{G(iii)(1)}] $U_{\spokej} \cap U_{\rim} \subseteq U_{\constrj}$ for $j=1,3$,
        \item[\namedlabel{itm:Giv-rho-2}{G(iii)(2)}] $U_{\rim}\subseteq U_{\constrOne}\cup U_{\constrThree}$, and
        \item[\namedlabel{itm:Giv-rho-3}{G(iii)(3)}] $U_{\constrOne}\cap U_{\constrThree}\cap U_{\rim} \subseteq 
        U_{\spokeTwo}$.
    \end{description}
\end{description}    

To complete the proof, we now show that \ref{itm:Giv-rho} and \ref{itm:Diii} imply \ref{itm:Wiv}.

Let $\realiz$ be a realization of $\code$ 
such that $U_{\rim}$ and $U_{\rim}\cap U_{\spokej}$ are convex for $j=1,2,3$. 
Then, for $j=1,3$, condition~\ref{itm:Diii} and the fact that $\scplexC$ is the nerve of $\realiz$ imply the inequality here:
\begin{align*} 
    \varnothing ~\neq~  (U_{\spokej}\cap U_{\rim} )
        ~\subseteq~ U_{\constrj}~,
\end{align*}
and the containment follows from~\ref{itm:Giv-rho-1}.
So, for
$j=1,3$, 
there exists
$p_j\in U_{\spokej}\cap U_{\rim} \subseteq
U_{\constrj}$. As $p_1$ and $p_3$ are in the convex set $U_{\rim}$, so too is the line segment, denoted by $L$, between $p_1$ and $p_3$:
\begin{align} \label{eq:proof-eq-2}
    L ~\subseteq~ U_{\scfaceT} ~\subseteq~ (U_{\constrOne}\cup U_{\constrThree})~,
\end{align}
and the second containment is by~\ref{itm:Giv-rho-2}. 
Thus, the (connected) set $L$ is covered by two nonempty sets 
$U_{\constrOne}\cap L$ and
$U_{\constrThree}\cap L$,
which are open in the subspace topology of $L$.  So, 
\begin{align} \label{eq:proof-eq-3}
\varnothing ~\neq~ (U_{\constrOne}\cap U_{\constrThree}\cap L) 
~\subseteq~
(U_{\constrOne}\cap U_{\constrThree}\cap U_{\rim})
~\subseteq~
U_{\spokeTwo}
~,
\end{align}
where the containments are by~\eqref{eq:proof-eq-2} and
\ref{itm:Giv-rho-3},
respectively.  
Now it follows 
from~\eqref{eq:proof-eq-2} and~\eqref{eq:proof-eq-3}
that $L$ meets  $U_{\spokeTwo}\cap U_{\rim}$, and so 
\ref{itm:Wiv} holds.
\end{proof} 

In the previous proof, we showed that \emph{every} line segment
whose endpoints are one in $U_{\spokeOne}\cap U_{\rim}$ and the other in
$U_{\spokeThree}\cap U_{\rim}$, meets $U_{\spokeTwo} \cap U_{\rim}$. This
is, on its face, stronger than~\ref{itm:Wiv}, which requires the
existence of only one such line segment. 

\begin{example}[A code with a sprocket] \label{ex:jeffsc2-again}
Recall the code 
$\code_2 =  \{ {\bf 1236}, 
    {\bf 234}, {\bf 135}, {\bf 456}, 
    13, 23, 4, 5, 6,
    \varnothing    \}$
from Example~\ref{ex:jeffsc2}.
We show that $\WheelFull =(5, 6, 4, 3)$ is a sprocket, with witnesses $\rho_1 = 13$ and $\rho_3 = 23$.  First, 
$\spokeOne \cup \spokeTwo \cup \spokeThree = 456 \in \scplexC$, 
and 
$\tk_{\code}(\spokei \cup \spokej) = \{456\} = \tk_{\code}(456)$ for all $1 \leq i < j \leq 3$.  So 
\ref{itm:Di} holds.  
Next, $\spokeOne \cup \spokeTwo \cup \spokeThree \cup \scfaceT = 3456 \notin \scplexC$, so 
\ref{itm:Dii} is satisfied.  
Also, 
$\spokeOne \cup \scfaceT =35$, 
$\spokeTwo \cup \scfaceT =36$, 
and 
$\spokeThree \cup \scfaceT =34$ are all faces of $\scplexC$; this verifies  
\ref{itm:Diii}.
Next, 
 $\tk_{\code}(\spokeOne \cup \rim)
 = \tk_{\code}(35)=\{135\} 
 \subseteq \{1236, 13, 135\}= 
 \tk_{\code}(13) =
 \tk_{\code}(\rho_1)$ and
 $\tk_{\code}(\spokeTwo \cup \rim) 
 = \tk_{\code}(34) = \{234\} 
 \subseteq \{23, 235\} = \tk_{\code} (23)
 =
 \tk_{\code}(\rho_2)$,
 which verifies~\ref{itm:Div-rho-1}.  
Checking \ref{itm:Div-rho-2} and~\ref{itm:Div-rho-3} 
is similarly straightforward.
\end{example}

However, not every wheel is a sprocket. 
\begin{example}[A wheel that is not a sprocket] \label{ex:code-tree-link}
Consider the neural code  
$$\code_{\treelink} = \{
{\bf 123},
{\bf 145},
{\bf 245},
{\bf 246},
{\bf 346},
24,45,46,
1,2,3,
\varnothing 
\}~,$$  
and consider 
$\Wheel_{\treelink} = (\spokeOne, \spokeTwo, \spokeThree, \rim)
= (1, 2, 3, 4)$.
Here, 
``TL'' stands for ``tree link'', and the meaning will be apparent in the proof of Proposition~\ref{prop:3-sparse-counterex}, where we show that $\Wheel_{\treelink}$ is a wheel.  
For now, we show that $\Wheel_{\treelink}$ is not a sprocket by proving there is no eligible pair $\rho_1$ and $\rho_3$.

Suppose for contradiction that $\rho_1$ and $\rho_3$ satisfy \ref{itm:Div-rho}. 
The codewords $145$ and $346$ contain, respectively, $\spokeOne \cup \rim = 14$ and $\spokeThree \cup \rim = 34$, so \ref{itm:Div-rho-1} implies that $\rho_1 \subseteq 145$ and $\rho_3 \subseteq 346$.
Next, $24$ is a codeword that contains $\scfaceT = 4$, so, 
by \ref{itm:Div-rho-2}, we have
$\rho_1 \subseteq 24$ or $\rho_3 \subseteq 24$. 
The above constraints imply that  $\rho_1 \subseteq \{4\}$ or $\rho_3 \subseteq \{4\}$.
Hence, $\rho_1 \cup \rho_3 \cup \scfaceT  =\rho_1 \cup \rho_3 \cup \{4\}$
is a subset of $145$ or $346$, and so $\tk_{\code}(\constrOne \cup \constrThree \cup \rim)$
contains 
$145$ or $346$. However, neither $145$ nor $346$ is in $\tk_{\code}(\{2\}) = \tk_{\code}(\spokeTwo)$, contradicting \ref{itm:Div-rho-3}.
Thus, $\Wheel_{\treelink}$ is not a sprocket of $\code_{\treelink}$. 
\end{example} 

As noted in the above example, we will show that 
the code $\code_{\treelink}$ has a wheel.  To do so, we need the following definition.

\begin{definition}
  \label{def:wireWheel} Let $\code$ be a neural code.
  A tuple $\WheelFull \in (\scplexC)^4$ is a \textit{wire wheel} of~$\code$ if the following hold:
  \begin{itemize}
    \item $\Wheel$ satisfies~\ref{itm:Di} and \ref{itm:Dii}, 
  \item $\scfaceT\not\in \code$,
  \item $\LkCfaceT$ is a tree,
  \item for $i=1,2,3$, the set $\spokei \smallsetminus \rim$ 
  has size one and 
  is a vertex of the tree $\LkCfaceT$,
  \item the unique path in the tree $\LkCfaceT$ between $\spokeOne \smallsetminus \rim$ and
    $\spokeThree \smallsetminus \rim$ contains $\spokeTwo
    \smallsetminus \rim$.
  \end{itemize}
The intuition behind this name is that wire wheels have thin spokes.
  \end{definition}
  
To prove that wire wheels are wheels, we need the following lemma.

\begin{lemma} \label{lem:wire-wheel-spokes}
If $\WheelFull $ is a wire wheel of a neural code $\code$, then the sets 
$\spokeOne \smallsetminus \rim$, 
$\spokeTwo \smallsetminus \rim$, 
and
$\spokeThree \smallsetminus \rim$ 
are distinct.
\end{lemma}

\begin{proof}
Assume for contradiction that 
$(\spokej \smallsetminus \rim)  = 
(\spokek \smallsetminus \rim)$, for some $1 \leq j < k \leq 3$.  It follows that $\spokej \cup \rim = \spokek \cup \rim$.  We also know that $\spokej \cup \rim \in \scplexC $, because $\spokej \smallsetminus \rim$ is a vertex of the link $\LkCfaceT$.  We conclude that $\spokej \cup \spokek \cup \rim = \spokej \cup \rim \in \scplexC $. Hence, there exists a codeword $c \in \code$ such that $c \in \tk_{\code}(\spokej \cup \spokek \cup \rim) $.  
This trunk $ \tk_{\code}(\spokej \cup \spokek \cup \rim) $ is contained in 
$\tk_{\code}(\spokej \cup \spokek )$, which, by~\ref{itm:Di}, equals
      $ \tk_{\code}(\spokeOne \cup \spokeTwo \cup \spokeThree)$.
We conclude that 
$c \in \tk_{\code}(\spokeOne \cup \spokeTwo \cup \spokeThree \cup \rim )$.
But this implies that $\spokeOne \cup \spokeTwo \cup
      \spokeThree \cup \rim \in \scplexC$, which contradicts~\ref{itm:Dii}.
\end{proof}

  \begin{prop}[Every wire wheel is a wheel]\label{prop:wire-wheel}
    If $\Wheel$ is a wire wheel of a neural code $\code$, then $\Wheel$ is a wheel of $\code$.
    Consequently, codes with wire wheels are non-convex.
\end{prop}

\begin{proof} 
  Let $\WheelFull$ be a wire wheel of $\code$. By the assumption that $\Wheel$ satisfies \ref{itm:Di} and \ref{itm:Dii}, and Proposition~\ref{prop:D}, we have that $\Wheel$ satisfies \ref{itm:Wi} and \ref{itm:Wii}. Thus it suffices to show that $\Wheel$ satisfies \ref{itm:Wiv}.

By Lemma~\ref{lem:wire-wheel-spokes}, the vertices $\spokeell \smallsetminus \rim$ of $\LkCfaceT$ are distinct.
We can therefore relabel the neurons, if necessary, so that
$\spokeell \smallsetminus \rim = \{ \ell \}$ for $\ell = 1, 2, 3$.
Now let $\fullreal$ be a realization of $\code$ 
such that $U_{\rim}$ and $U_{\rim}\cap U_{\spokej}$ are convex for $j=1,2,3$. Note that since $\{\ell \} \cup \rim = \spokeell\cup \rim$ for
$\ell = 1, 2, 3$, we have that $U_{\spokeell} \cap U_{\rim} = U_{\ell} \cap U_{\rim}$. Thus it suffices to show that there is a line segment
        from a point in $U_1 \cap U_{\rim}$ to a point in $U_3 \cap U_{\rim}$ that meets $U_2 \cap U_{\rim}$.  
        
        First, as $1,3 \in \LkCfaceT$, the sets $U_1 \cap U_{\rim}$ and $U_3 \cap U_{\rim}$
        are nonempty, so let $L$ be a line segment from a point in one set to a point in the other.
        The endpoints of $L$ are in the convex set $U_{\rim}$ and so $L$ is contained in $U_{\rim}$.
        As $\rim \notin \code$, the line segment $L$ is covered by the relatively open intervals $U_{j} \cap L$, where $j \in \LkCfaceT$.  Additionally, as the link is one-dimensional, these intervals have only pairwise intersections.  Hence, for each such interval $U_{j} \cap L$ that is nonempty, there exists a point $p_{j}$ that is in that interval and no other intervals.  We conclude that the intersection patterns  of the intervals correspond to a path in the link, which by assumption contains the vertex $2$.  Hence, the desired point $p_2$ in $U_j \cap L$ exists.
\end{proof}

\begin{remark}
\label{rem:order-forcing}
The approach used at the end of the proof of Proposition~\ref{prop:wire-wheel} is called \textit{order-forcing}, and is described further in \cite{jeffs2020order}.  Another application of order-forcing is found in~\cite{chan-nondegenerate}.
\end{remark}

\begin{remark}[One-dimensional links] \label{rem:1-d-link}
Wire wheels involve a non-codeword 
$\rim \in \scplexC\smallsetminus
\code$ for which the link $\LkCfaceT$ is a tree (which is contractible).  If this link is not a tree but still is 
one-dimensional, then the link is non-contractible and so, by Proposition~\ref{prop:summary-prior-results}, the code $\code$ is non-convex due to a local obstruction.  
Proposition~\ref{prop:wire-wheel} is therefore notable for being able to detect cases in which some $\rim$ fails to generate a local obstruction and yet the code is still non-convex.
\end{remark}

Next, 
we use 
Proposition~\ref{prop:wire-wheel} to
answer
a question prompted by recent work on neural codes. A \textit{3-sparse} code is a code for which the neural complex has
dimension at most 2. The authors of \cite{new-obs} 
asked whether every 3-sparse code with no local obstructions, is
convex. 
(The answer is ``yes'' for codes on up to 5 neurons \cite{what-makes,goldrup2020classification}; see Proposition~\ref{prop:five-neur-top-dim} in a later section.)
We answer this question in the negative, by
providing an example of a 
code on 6 neurons that has a wire wheel.

\begin{prop}\label{prop:3-sparse-counterex}
There is a non-convex 3-sparse code with no local obstructions.
\end{prop}

\begin{proof}
We return to the code $\code_{\treelink} := 
\{
{\bf 123},
{\bf 145},
{\bf 245},
{\bf 246},
{\bf 346},
24,45,46,
1,2,3,
\varnothing
\}$
from Example~\ref{ex:code-tree-link}.  
(Recall that $\treelink$ stands for ``tree link''.)
The maximal codewords have length 3, so $\dim(\scplex(\code_{\treelink})) = 2$. The max-intersection faces are $\varnothing, 1,2,3,4,24,45$, and $46$. With the exception of $4$, all of these intersections are codewords of $\code_{\treelink}$. While $4\not\in\code$, the link 
$\Lk{\{4\}}{\scplex(\code_{\treelink})}$
is the following path, which is contractible:
\begin{tikzpicture}[scale=0.5]
        \vertex[label=$1$](p1) at (0,0) {};
        \vertex[label=$5$](p2) at (1,0) {};
        \vertex[label=$2$](p3) at (2,0) {};
        \vertex[label=$6$](p4) at (3,0) {};
        \vertex[label=$3$](p5) at (4,0) {};
        \path [-](p1) edge node[above] {} (p2);
        \path [-](p2) edge node[above] {} (p3);
        \path [-](p3) edge node[above] {} (p4);
        \path [-](p4) edge node[above] {} (p5);
\end{tikzpicture}.
Thus, by 
Proposition~\ref{prop:summary-prior-results}, 
the code 
$\code_{\treelink}$ has no local obstructions. 

Consider $\Wheel_{\treelink} := 
(\spokeOne, \spokeTwo, \spokeThree, \rim) = (1, 2, 3, 4)$.
First, 
$\Wheel_{\treelink}$ satisfies \ref{itm:Di}, 
as $123 \in \scplex(\code_{\treelink}) $
and 
$ \tk_{\code_{\treelink}}(\spokei \cup \spokej)
= \{123\} = \tk_{\code_{\treelink}}(\spokeOne \cup \spokeTwo \cup \spokeThree)$
for $1 \leq i < j \leq 3$.
Next, $\spokeOne \cup \spokeTwo \cup \spokeThree \cup \rim = 1234 \notin \code_{\treelink}$, so \ref{itm:Dii} also holds.  
Finally, 
we already saw that the link 
$\Lk{\{4\}}{\scplex(\code_{\treelink})}$
is a path (and thus a tree) in which the unique path from vertex 
$( \spokeOne \smallsetminus \scfaceT ) =1$ to 
$( \spokeThree \smallsetminus \scfaceT ) =3$ passes through the vertex 
$( \spokeTwo \smallsetminus \scfaceT ) =2$.
Hence, $\Wheel_{\treelink}$ is a wire wheel, 
and so Proposition~\ref{prop:wire-wheel} implies that $\code_{\treelink}$ is non-convex.
\end{proof}

Recall that, in Definition~\ref{def:wheel}, we distinguished between a wheel of a code and a wheel of a realization. 
Being a wheel of a code means being a common wheel of every realization of the code. 
However, 
for proving that certain codes are not convex,
this is 
too strong a requirement. Indeed, it suffices to show that every
realization
has a wheel, which may vary from one realization to another. 
Accordingly, 
we 
now introduce a different type of combinatorial wheel using this more
flexible approach.

\begin{definition} \label{def:wheelFrame}
  Let $\code$ be a code 
  with neural complex 
  $\scplexC$.
  A triple
$(\spokeOne,\spokeThree, \rim) \in (\scplexC)^3$ is a \textit{wheel frame}
of $\code$ if it satisfies the following conditions: 
\begin{description}
     \item[\namedlabel{itm:D1iii}{F(i)}]$\spokeOne\cup \spokeThree\in \scplexC$, and for all $\omega\subseteq \spokeOne\cup \spokeThree$ such that 
	    \begin{enumerate}
	       \item[(1)] \label{D1iii1} neither $\omega \subseteq \spokeOne$ nor $\omega
                 \subseteq \spokeThree$, and
	       \item[(2)] \label{D1iii2} $\omega \cup \rim \in \scplexC$,
            \end{enumerate}
        it follows that 
            $\tk_{\code} (\spokeOne \cup \omega)
            =
            \tk_{\code}(\spokeOne \cup \spokeThree)$ 
            and 
            $\tk_{\code} (\spokeThree \cup \omega)
            =
            \tk_{\code}(\spokeOne \cup \spokeThree)$; 
	\item[\namedlabel{itm:D1iv}{F(ii)}] $\rim\cup \spokeOne\cup \spokeThree\not\in \scplexC$;
    \item[\namedlabel{itm:D1i}{F(iii)}] 
        $\spokeOne\cup \rim\in \scplexC$ and $\spokeThree\cup \rim\in \scplexC$; and
    \item[\namedlabel{itm:D1ii}{F(iv)}]
        $\spokeOne \cap \spokeThree = \varnothing$ and 
        $\tk_{\code}(\scfaceT) \subseteq \cup_{i \in \spokeOne \cup \spokeThree} \tk_{\code}(\{i\})$. 
\end{description}
\end{definition}

\begin{theorem}[Wheel frames generate wheels]\label{prop:wire-frame}
  Let $(\spokeOne,\spokeThree, \rim)$ be a wheel frame of a code
  $\code$ on $n$ neurons. Then, for every realization $\realiz$ of $\code$, there exists
  $\scfaceS_2 \subseteq [n]$ 
  such that $(\spokeOne,\spokeTwo,\spokeThree,\rim)$ is a wheel of
  $\realiz$. Consequently, codes with wheel frames are non-convex.
\end{theorem} 

\begin{proof} Let $\fullreal$ be a realization of $\code$.  Our first task is to construct $\spokeTwo$.

We consider two cases.  
If $U_{\rim}$
is non-convex, 
set
$\spokeTwo := ( \spokeOne \cup \spokeThree )$. 
(In this case, 
the non-convexity of $U_{\rim}$ implies that $\WheelFull$ vacuously satisfies 
\ref{itm:Wiv} with respect to $\realiz$.) 

Now consider the remaining case, that is, when $U_{\rim}$ is convex.
  As $\scplexC$ is the nerve of $\realiz$, 
  \ref{itm:D1iv} and~\ref{itm:D1i} 
  imply that $U_{\spokeOne\cup \rim}$ and
  $U_{\spokeThree\cup \rim}$ are disjoint, nonempty sets.
Let $p_1\in U_{\spokeOne\cup \rim}$ and $p_3\in U_{\spokeThree\cup \rim}$, and 
let $L$ denote the line segment with endpoints $p_1$ and
$p_3$. Then $L\subseteq
U_{\rim}$, because 
$p_1$ and $p_3$ are both in the convex set $U_{\rim}$.

We claim that there exists some $p_2\in L$ between $p_1$ and $p_3$
such that $p_2\not\in U_{\spokej}$ for $j=1,3$. Set $L_j = L\cap
(\bigcup_{i\in \spokej}U_i)$ for $j=1, 3$. By \ref{itm:D1ii}, we have that
$U_{\rim}$ (and thus $L$) is covered by $\{U_i\}_{i\in \spokeOne \cup
  \spokeThree}$, so $L = L_1 \cup L_3$. Both $L_1$ and
$L_3$ are relatively open subsets of the connected set $L$,
and $L_1, L_3 \neq
\varnothing$ (because $p_1\in L_1$ and $p_3\in L_3$), so
it follows that $L_1\cap L_3\neq
\varnothing$. 
We conclude that there exist $i_1\in \spokeOne$, $i_3\in
\spokeThree$, and $p_2\in L \subseteq U_{\rim}$ such that $p_2\in U_{i_1} \cap
U_{i_3}$.  We remark that 
$i_1 \in (\spokeOne \smallsetminus \spokeThree)$ and 
$i_3 \in (\spokeThree \smallsetminus \spokeOne)$, 
because $\spokeOne \cap \spokeThree = \varnothing$ (by~\ref{itm:D1ii}).
Let $\spokeTwo := \{i_1 \} \cup \{i_3\}$.

We now claim that $\WheelFull$ is a
wheel of $\realiz$. 
To see this, we first observe that, by construction, $\Wheel$ satisfies~\ref{itm:Wiv} with respect to $\realiz$.   
To complete the proof, we show that $\Wheel$ also satisfies \ref{itm:Di} and \ref{itm:Dii}
(recall Proposition~\ref{prop:D}).

We first show \ref{itm:Di}. 
By construction, 
$\spokeTwo \subseteq (\spokeOne \cup \spokeThree)$. So, 
$\spokeOne \cup \spokeThree= \spokeOne \cup \spokeTwo \cup \spokeThree $ holds, 
which implies the following:
\begin{description}
     \item[\namedlabel{itm:part-a}{(a)}]
     $\spokeOne \cup \spokeTwo \cup \spokeThree$ is in $\scplexC$ (due to \ref{itm:D1iii}), and 
     \item[\namedlabel{itm:part-b}{(b)}]
    $\tk_{\code}(\spokeOne \cup \spokeThree) = \tk_{\code}(\spokeOne \cup \spokeTwo \cup \spokeThree)$. 
\end{description}
Hence, only two trunk conditions in~\ref{itm:Di} are left to verify, and by~\ref{itm:part-b} these conditions are equivalent to the following:
\begin{align} \label{eq:trunk-condition}
 \tk_{\code}(\spokei \cup \spokeTwo) ~=~ \tk_{\code}(\spokeOne \cup \spokeThree) \quad \textrm{for~} i=1,3~. 
\end{align}
for $i=1,3$. 
If $U_{\rim}$ is non-convex, then the equalities~\eqref{eq:trunk-condition} follow from 
$\spokeTwo = ( \spokeOne \cup \spokeThree )$.  Now consider the case when $U_{\rim}$ is convex.
We prove~\eqref{eq:trunk-condition} by 
applying~\ref{itm:D1iii} 
with $\omega = \spokeTwo$, as follows.  
We already saw that $\spokeTwo \subseteq (\spokeOne \cup \spokeThree)$.  
Also, $\spokeTwo \nsubseteq \spokeOne$ and 
$\spokeTwo \nsubseteq \spokeThree$ hold, because 
$i_1 \in (\spokeTwo \smallsetminus \spokeThree)$ and
$i_3 \in (\spokeTwo \smallsetminus \spokeOne)$. 
Lastly, since
$p_2 \in U_{\spokeTwo} \cap U_{\rim}$ we see that $U_{\spokeTwo}\cap
U_{\rim}\neq \varnothing$, and so (as $\scplexC$ is the nerve of $\realiz$) we have $\spokeTwo \cup \rim\in
\scplexC$. 
Thus, \ref{itm:D1iii} implies the equalities~\eqref{eq:trunk-condition}.  Hence, $\Wheel$ satisfies \ref{itm:Di}.

Finally, $\Wheel$ satisfies \ref{itm:Dii} 
because of~\ref{itm:D1iv},
since $\spokeTwo \subseteq (\spokeOne \cup \spokeThree)$ implies $\rim
\cup \spokeOne \cup \spokeThree = \rim \cup \spokeOne \cup \spokeTwo
\cup \spokeThree$. 
\end{proof}

Next, we show that the code $\code^{\star}$ from Example~\ref{ex:lienshiu} has a wheel frame.

\begin{proposition}[$\code^{\star}$ has a wheel frame]
 \label{prop:original-counterexample-wheel-frame} 
Let
$\code^{\star}$ be the neural code from Example~\ref{ex:lienshiu}. 
If $\calD$ is a code such that
\begin{description}
     \item[\namedlabel{item:code-1}{(1)}] 
     		$\code^{\star}  \subseteq \calD \subseteq \scplexCstar$ and
     \item[\namedlabel{item:code-2}{(2)}] 
		$\calD$ does \uline{not} contain any of the following codewords: $1$, $234$, and $245$,
\end{description}
then $\calD$ contains a wheel frame and thus is non-convex.
In particular, $\code^{\star}$ has a wheel frame. 
\end{proposition}

\begin{proof}
Assume that $\calD$ satisfies \ref{item:code-1} and~\ref{item:code-2}.  It follows that $\scplexD=\scplexCstar$.

We show that $(\spokeOne, \spokeThree, \rim) = (23, 45, 1)$ is a wheel frame, as follows.  
It is straightforward to check that conditions~\ref{itm:D1iv} and~\ref{itm:D1i} hold.  Condition~\ref{itm:D1ii} is also easy to check (here, the assumption $1 \notin \calD$ is used).  Finally we consider~\ref{itm:D1iii}. 
The only set $\omega \subseteq (\spokeOne \cup \spokeThree) = 2345$ for which 
 $\omega \nsubseteq \spokeOne  = 23$, 
 $\omega \nsubseteq \spokeThree = 45$, and $\omega \cup \scfaceT = \omega \cup \{1\} \in \scplexCstar$ 
 is the set $\omega=34$.  For this set $\omega$, the trunk conditions in~\ref{itm:D1iii}, namely, 
            $\tk_{\code} (\spokeOne \cup \omega)
            =
            \tk_{\code}(\spokeOne \cup \spokeThree)$ 
            and 
            $\tk_{\code} (\spokeThree \cup \omega)
            =
            \tk_{\code}(\spokeOne \cup \spokeThree)$, 
            are readily seen to hold, as (respectively) $234 \notin \calD$ and $345 \notin \calD$.
\end{proof}

We end this section with another example of a code with a wheel frame.

\begin{example} \label{ex:wheel-frame-trivial}
Consider the neural code $\code= \{ {\bf 12}, {\bf 13}, {\bf 23}, 1,2, \emptyset \} $. 
It is straightforward to check that $(\spokeOne, \spokeThree, \rim):= (1,2,3)$ is a wheel frame of $\code$, and so $\code$ is non-convex.   Moreover, by following the proof of Theorem~\ref{prop:wire-frame}, we see that $\WheelFull = (1, 12, 2,3)$ is a wheel of {\em every} realization of $\code$ and hence is a wheel of $\code$ itself. 

We make several observations about this wheel $\Wheel$.  First, \ref{itm:Diii}
(or equivalently \ref{itm:Wiii}) does not hold: $\spokeTwo \cup \tau = 123 \notin \scplexC$. 
Thus, the non-convexity of $\code$ comes from the fact
that $\Wheel$ only vacuously satisfies~\ref{itm:Wiv} (recall Remark~\ref{rmk:extra-condition-not-in-def}). 
Indeed, in every realization $\realiz = \{U_1,U_2,U_3\}$ of $\code$, the set $U_{\rim}= U_3$ is disconnected and hence non-convex:
$U_3 = (U_3 \cap U_1) \cup (U_3 \cap U_2)$, where 
$(U_3 \cap U_1)$ and $(U_3 \cap U_2)$ are open, nonempty, and disjoint. 
This non-convexity can also be seen through local obstructions: the mandatory face $3$ of $\scplexC$ is a non-codeword of $\code$, and so $\code$ has a local obstruction and hence is non-convex (Proposition~\ref{prop:summary-prior-results}).
\end{example}

We revisit the code in Example~\ref{ex:wheel-frame-trivial} in the next section.

\section{Refining the search for wheels} \label{sec:search-2}
In the process of checking whether a code $\code$ has a wheel, it is natural to ask whether we must consider every quadruple $(\spokeOne,\spokeTwo,\spokeThree,\rim)$ of faces
in $\scplexC$. 
Fortunately, the answer is ``no''.  This section features several results in this direction. 
First, for sprockets and wire wheels, there are no containments among the spokes 
 $\spokei$ (Corollary~\ref{cor:no-spoke-containment-sprocket-wire-wheel}) 
 and the rim $\rim$ is never a codeword of~$\code$
 (Propositions~\ref{prop:missingrim} and~\ref{prop:wire-wheel-max-intersection}). 
Further constraints on rims are motivated by local obstructions, which we recall can be detected at the level of $\scplexC$, rather than a restriction of $\scplexC$, through some max-intersection face.  
We show that combinatorial wheels also 
``bubble up'', that is, a code with a sprocket, wire wheel, or wheel frame also has a sprocket, wire wheel, or wheel frame in which the rim $\rim$ is a max-intersection face (Propositions~\ref{prop:bubble}, \ref{prop:wire-wheel-max-intersection}, and~\ref{prop:wheel-frame-max-intersection}). Finally, we conjecture that this bubble-up property 
generalizes to all wheels (Conjecture~\ref{conj:bubbleup}).  

\subsection{Precluding containments among spokes} \label{sec:no-containments-spokes}
Consider a wheel 
of a code $\code$ that violates~\ref{itm:Diii} (for instance, the wheel in Example~\ref{ex:wheel-frame-trivial}).  
We have seen that such a wheel 
is ``extraneous'' in the sense that, 
in \uline{every} realization of $\code$, 
condition \ref{itm:Wiv} holds vacuously 
(recall Remark~\ref{rmk:extra-condition-not-in-def}).  
We now show that, for non-extraneous wheels, 
there are no containments among the three spokes $\spokei$. 

\begin{prop}[No containments among spokes if~\ref{itm:Diii} holds]\label{prop:sigmacontain}
Let $\WheelFull$ be a wheel of a neural code $\code$.
If $\Wheel$ satisfies~\ref{itm:Diii}, 
then $\spokej \not\subseteq \spokek$ for 
distinct $j,k \in \{1,2,3\}$.
 \end{prop}
 
 \begin{proof}
Assume for contradiction that $\Wheel$ is a wheel of $\code$ such 
that~\ref{itm:Diii} holds and also $\spokej \subseteq \spokek$ for some 
$j\neq k$. Note that \ref{itm:Di} and \ref{itm:Dii} also hold, by Proposition~\ref{prop:D}. 

By~\ref{itm:Diii}, we have that $\scfaceS_k \cup \scfaceT \in \scplexC$. 
So, there exists $c\in \code$ such that $\spokek \cup \rim \subseteq c$, and so $\spokej\cup \rim\subseteq c$ also holds (because $\spokej\subseteq \spokek$). Thus, $c$ contains $\spokej \cup \spokek$. So, $c\in \tk_{\code}(\spokej \cup \spokek) = \tk_{\code}\hub$, where the equality is by~\ref{itm:Di}. 
Hence, $c \in \tk_{\code}(\spokeOne \cup \spokeTwo \cup \spokeThree \cup \rim)$, which implies that $\spokeOne \cup \spokeTwo \cup \spokeThree \cup \rim \in \scplexC$. But this contradicts \ref{itm:Dii}.
 \end{proof}

Next, we apply Proposition~\ref{prop:sigmacontain} to sprockets and wire wheels.

\begin{corollary}[No containments among spokes of sprockets and wire wheels]\label{cor:no-spoke-containment-sprocket-wire-wheel}
If $\WheelFull$ is a sprocket or a wire wheel of a neural code $\code$, 
then $\spokej \not\subseteq \spokek$ for 
distinct $j,k \in \{1,2,3\}$.
\end{corollary}

\begin{proof}
Sprockets and wire wheels are wheels (Proposition~\ref{prop:sprocket} and~\ref{prop:wire-wheel}).  So, by 
Proposition~\ref{prop:sigmacontain}, we need only show that sprockets and wire wheels satisfy~\ref{itm:Diii}.  Sprockets satisfy~\ref{itm:Diii} by definition. 
Now assume $\WheelFull$ is a wire wheel.  By definition, $\spokeell \smallsetminus \rim$, for $ \ell =1,2,3$, is a vertex of the link $\LkCfaceT$.  Hence, $\spokeell \cup \rim$ is a face of $\scplexC$, and so~\ref{itm:Diii} holds.
\end{proof}

Corollary~\ref{cor:no-spoke-containment-sprocket-wire-wheel} 
does not extend to wheels generated by wheel frames.  Indeed, we saw such a wheel in Example~\ref{ex:wheel-frame-trivial} (in that wheel, $\spokeOne = 1 \subseteq \spokeTwo = 12$).  
Moreover, that example showed that, 
for wheels generated by wheel frames, \ref{itm:Diii} need not hold and so Proposition~\ref{prop:sigmacontain} does not apply.  
Nevertheless, the ``outer spokes'' $\spokeOne$, $\spokeThree$ of a wheel frame are guaranteed to \uline{not} contain each other, as follows.

\begin{proposition}[No containments among spokes $\spokeOne$ and $\spokeThree$ of  wheel frames]\label{prop:no-spoke-1-3-containment-wheel-frame}
If $(\spokeOne, \spokeThree, \rim)$ is a wheel frame of a code $\code$, then 
$\spokeOne \nsubseteq \spokeThree$
and 
$\spokeThree \nsubseteq \spokeOne$. 
\end{proposition}

\begin{proof}
\ref{itm:D1ii} implies that $\spokeOne \cap \spokeThree = \emptyset$, so (by symmetry) it suffices to show that $\spokeOne \neq \emptyset$.  Indeed, if $\spokeOne = \emptyset$, then \ref{itm:D1iv} and~\ref{itm:D1i} can not both hold.
\end{proof}

\subsection{Constraints on sprockets} \label{sec:sprocket-constraints}

The results in this subsection simplify the search for sprockets.   Our first result states that we need only check those quadruples $(\spokeOne,\spokeTwo,\spokeThree,\rim)$ for which $\rim\not\in \code$.

 \begin{prop}[Rims of sprockets are non-codewords]\label{prop:missingrim}
If $\WheelFull$ is a sprocket of a neural code $\code$, then $\rim\not\in \code$.
 \end{prop}
 
 \begin{proof}
Assume for contradiction that $\WheelFull$ is a sprocket of $\code$ with $\rim\in \code$. 
By definition and Proposition~\ref{prop:D}, $\Wheel$ satisfies
 \ref{itm:Di}, \ref{itm:Dii}, and~\ref{itm:Diii}.
Let $\constrOne$, $\constrThree$ 
be witnesses for $\Wheel$, as in~\ref{itm:Div-rho}.  

Next, $\rim \in \code$ implies that $\rim \in \tk_{\code}(\rim)$. Hence, by \ref{itm:Div-rho-2}, 
we have $\rim \in \tk_{\code}(\constrOne)$ or $\rim \in \tk_{\code}(\constrThree)$. 
By symmetry, we may assume that $\rim \in \tk_{\code}(\constrOne)$ or, equivalently, $\constrOne \subseteq \rim$.
  
By \ref{itm:Diii}, there exists a codeword $c \in \code$ 
such that $c\in \tk_{\code}(\spokeThree \cap \rim)$. 
Then by \ref{itm:Div-rho-1} we have 
$c\in \tk_{\code}(\constrThree)$, and thus 
$c\in \tk_{\code}(\constrThree \cup \rim)$. 
It follows that 
$c\in \tk_{\code}(\constrOne \cup \constrThree \cup \rim)$
(because $\constrOne \subseteq \rim$), 
which, by \ref{itm:Div-rho-3}, implies that $c\in \tk_{\code}(\spokeTwo)$.
  
Thus, $c\in \tk_{\code}(\spokeTwo)\cap \tk_{\code}(\spokeThree) = \tk_{\code}(\spokeTwo \cup \spokeThree)$. 
Then \ref{itm:Di} implies that 
$c\in \tk_{\code}\hub$.
But then 
$c\in \tk_{\code}\hub \cap \tk_{\code}(\rim) = \tk_{\code}(\spokeOne \cup \spokeTwo \cup \spokeThree \cup \rim)$,
which contradicts \ref{itm:Dii}.
 \end{proof}

The next result states that the witnesses $\constrOne$ and $ \constrThree$ for a sprocket, from~\ref{itm:Div-rho} in the definition of a sprocket, can not be equal.

\begin{prop}[Witnesses for a sprocket are distinct]\label{prop:rhosnotequal}
Let $\Wheel $ be a sprocket of a neural code $\code$.
If $\constrOne, \constrThree \in \scplexC $ are witnesses for the sprocket $\Wheel$ (as in~\ref{itm:Div-rho}), then $\constrOne \neq \constrThree$. 
 \end{prop}
 
\begin{proof}
Assume for contradiction that
$\constrOne, \constrThree$ are witnesses for the sprocket 
$\WheelFull$, and that $\constrOne= \constrThree$.  
Let 
$\constrBad:=\constrOne= \constrThree$. 
By \ref{itm:Diii}, there exists 
$c\in \tk_{\code}(\spokeOne \cup \rim)$. Thus, by \ref{itm:Div-rho-1}, we have that 
$c\in \tk_{\code}(\constrBad)$. So, 
$c\in \tk_{\code}(\constrBad \cup \rim)$, and thus by \ref{itm:Div-rho-3} 
(and the fact that $\constrBad=\constrOne= \constrThree$)
we have that 
$c\in \tk_{\code}(\spokeTwo)$. Furthermore, 
$c\in \tk_{\code}(\spokeOne) \cap \tk_{\code}(\spokeTwo) = \tk_{\code}(\spokeOne \cup \spokeTwo) = \tk_{\code}(\scfaceS_1 \cup \scfaceS_2 \cup \scfaceS_3)$, where the last equality is from \ref{itm:Di}. Hence, $c\in \tk_{\code}(\spokeOne \cup \spokeTwo \cup \spokeThree) \cap \tk_{\code}(\spokeOne \cup \rim) = \tk_{\code}(\spokeOne \cup \spokeTwo \cup \spokeThree \cup \rim)$, and thus $(\spokeOne \cup \spokeTwo \cup \spokeThree \cup \rim) \subseteq c$, which contradicts \ref{itm:Dii}.
 \end{proof}

The next result is inspired by the theory of local obstructions. Recall that 
 a code~$\code$ has no local obstruction if and only if $\code$ contains all mandatory faces of $\scplexC$  (Proposition~\ref{prop:summary-prior-results}). 
The reason behind this is a bubbling-up property: if $\code$ has a local obstruction $(\locobOne, \locobTwo)$, then it also has a local obstruction of the form $(\widetilde{\locobOne},[n]\smallsetminus \widetilde{\locobOne})$ where $\locobOne\subseteq \widetilde{\locobOne}$ and $\widetilde{\locobOne}$ is a mandatory face (and thus a max-intersection face) of $\scplexC$~\cite{what-makes}.  
In a similar manner, the following result shows that if a code $\code$ has a sprocket, then it also has a sprocket in which the rim $\rim$ is a max-intersection face of~$\scplexC$. 

\begin{prop}[Bubble-up property for sprockets]\label{prop:bubble}
If a neural code $\code$ has a sprocket,
then 
$\code$ has a sprocket $\WheelFull$ 
in which 
$\rim$ is a max-intersection face of $\scplexC$.
%
 \end{prop}

\begin{proof}
 Let $\WheelFull$ be a sprocket of a code $\code$, and let $\widetilde{\rim}$ be the intersection of all maximal codewords of $\code$ that contain $\rim$. It suffices to show that $\widetilde{\Wheel} := (\spokeOne,\spokeTwo,\spokeThree,\widetilde{\rim})$ is also a sprocket of $\code$. First, 
  \ref{itm:Di} depends only on $\spokeOne$, $\spokeTwo$, $\spokeThree$; 
  so the fact that 
  $\Wheel$
 satisfies \ref{itm:Di}
 implies that 
  $\widetilde{\Wheel}$ 
  does too.

 
 Next, we show that $\widetilde{\Wheel}$ satisfies \ref{itm:Dii}. The containment $\rim\subseteq \widetilde{\rim}$ implies that $\hub\cup \rim \subseteq \hub \cup \widetilde{\rim}$. Then, from $\hub\cup \rim\not\in \scplexC$ (because $\Wheel$ satisfies \ref{itm:Dii}), we have that $\hub\cup \widetilde{\rim}\not\in \scplexC$.
 
 To show that $\widetilde{\Wheel}$ satisfies \ref{itm:Diii}, we must show that $\spokej\cup \widetilde{\rim}\in \scplexC$ for $j=1,2,3$. 
 Because 
 $\Wheel$ satisfies \ref{itm:Diii}, 
 there exists a maximal codeword $c_j$ of $\code$ such that 
 $\spokej\cup \rim \subseteq c_j$.
 By construction of $\widetilde{\rim}$, we have that 
 $\widetilde{\rim} \subseteq c_j$. 
 Hence, $\spokej\cup \widetilde{\rim} \subseteq c_j$ and so, as desired, $\spokej\cup \widetilde{\rim}\in \scplexC$.
 
 To show that $\widetilde{\Wheel}$ satisfies \ref{itm:Div-rho}, we begin by choosing witnesses $\constrOne, \constrThree$ for the sprocket $\Wheel$ (as in~\ref{itm:Div-rho}).  Next, the containment $\rim\subseteq \widetilde{\rim}$ yields the following:
 \begin{itemize}
     \item $\tk_{\code}(\spokej \cup \widetilde{\rim})\subseteq \tk_{\code}(\spokej \cup \rim)$,
     \item $\tk_{\code}(\widetilde{\scfaceT})\subseteq \tk_{\code}(\scfaceT)$, and
     \item $\tk_{\code}(\rho_1 \cup \rho_3 \cup \widetilde{\scfaceT}) \subseteq \tk_{\code}(\rho_1 \cup \rho_3 \cup \scfaceT) $.
 \end{itemize}  
Combining these three containments with, respectively, the containments in~\ref{itm:Div-rho}(1--3) satisfied by $\Wheel$ by way of $\constrOne, \constrThree$, we conclude that $\constrOne, \constrThree$ are also witnesses for $\widetilde{\Wheel}$.  That is, $\widetilde{\Wheel}$ satisfies \ref{itm:Div-rho}.
\end{proof}
 
Propositions \ref{prop:missingrim} and \ref{prop:bubble} together imply the following result.
%
\begin{corollary} \label{cor:tau-is-max}
If a neural code $\code$ has a sprocket, then 
$\code$ has a sprocket $\WheelFull$ such that $\rim\not\in \code$ and $\rim$ is a max-intersection face of $\scplexC$.
\end{corollary}

\subsection{Beyond sprockets} \label{sec:beyond-sprocket}

A natural question is whether Corollary~\ref{cor:tau-is-max} extends beyond sprockets. 
In fact, a slightly stronger result holds for wire wheels (Proposition~\ref{prop:wire-wheel-max-intersection}), and a bubble-up property holds for wheel frames (Proposition~\ref{prop:wheel-frame-max-intersection}).
However, for wheel frames we do not know whether the rim $\rim$ is always a non-codeword.

\begin{proposition}[Properties of wire wheels] \label{prop:wire-wheel-max-intersection}
If $\WheelFull$ is a wire wheel of a neural code $\code$, then $\rim\not\in \code$ and $\rim$ is a max-intersection face of $\scplexC$.
\end{proposition}
\begin{proof}
Let $\WheelFull$ be a wire wheel of $\code$.  By definition, $\rim\not\in \code$.  Now we must show that $\rim$ is a max-intersection face of $\scplexC$.  

By definition, the link $\LkCfaceT$ is a tree, and 
the unique path in this tree between $\spokeOne \smallsetminus \rim$ and
    $\spokeThree \smallsetminus \rim$ contains $\spokeTwo
    \smallsetminus \rim$.  Denote this path by $p_1= (\spokeOne \smallsetminus \rim), p_2, \dots, p_r= (\spokeTwo \smallsetminus \rim), p_{r+1}, \dots, p_s= (\spokeThree \smallsetminus \rim)$. By Lemma~\ref{lem:wire-wheel-spokes}, we know that $1 < r <s$.  
    
    Next, we claim that $2 < r$ (and so, by symmetry, $r+1 < s$).  Assume for contradiction that $2=r$.  Then $\rim \cup \spokeOne \cup \spokeTwo = \rim \cup \{p_1, p_2\}$, which is a face of $\scplexC$ 
    because  $\{p_1, p_2\}$ is an edge of the link $\LkCfaceT$.  We also know that $\rim \cup \spokeOne \cup \spokeTwo = \rim \cup \spokeOne \cup \spokeTwo \cup \spokeThree$ (by~\ref{itm:Di}). We conclude that $\rim \cup \spokeOne \cup \spokeTwo \cup \spokeThree \in \scplexC$, which contradicts~\ref{itm:Dii}. Therefore, as claimed, $2 < r$ (and so $r+1 < s$).  In particular, $s > 4$.

Let $F_1:=
\rim \cup \{p_1,p_2\}$ and $
F_2:= \rim \cup \{p_{s-1}, p_s\}$.  Both $F_1$ and $F_2$ are facets of $\scplexC$, because $\{p_1,p_2\}$ and $ \{p_{r-1},p_r\}$ are edges of the link $\LkCfaceT$ (which is a tree).  Also, as $s > 4$, we have $F_1 \cap F_2 = \rim$.  So, $\rim$ is the intersection of two facets of $\scplexC$ and therefore is a max-intersection face.
\end{proof}

\begin{proposition}[Bubble-up property for wheel frames] \label{prop:wheel-frame-max-intersection}
If a neural code $\code$ has a wheel frame,
then 
$\code$ has a wheel frame
$(\spokeOne, \spokeThree, \rim)$ 
in which 
$\rim$ is a max-intersection face of $\scplexC$.
\end{proposition}
\begin{proof}
Let $ \mathcal{F} =(\spokeOne, \spokeThree, \rim)$ be a wheel frame of $\code$.  Let $\widetilde{\rim}$ be the intersection of all facets of $\scplexC$ that contain $\rim$ (in particular, $\rim \subseteq \widetilde{\rim}$). 
We must show that $\widetilde{ \mathcal{F}} = (\spokeOne,\spokeThree,\widetilde{\rim})$ is also a wheel frame of $\code$. First,~\ref{itm:D1iii} for $\widetilde{ \mathcal{F}} $ is implied by the same condition for $ \mathcal{F}$ (because $\rim \subseteq \widetilde{\rim}$).  Similarly,
$\widetilde{ \mathcal{F}}$ also inherits~\ref{itm:D1iv} from $ \mathcal{F}$, because $\rim \subseteq \widetilde{\rim}$ and simplicial complexes are closed with respect to containment.  Next, we know that
$\spokeOne\cup \rim\in \scplexC$ and $\spokeThree\cup \rim\in \scplexC$, because $ \mathcal{F}$ satisfies~\ref{itm:D1i}.  So, there exist facets $F_1$ and $F_3$ of $\scplexC$ such that 
$\spokeOne\cup \rim \subseteq F_1$ and 
$\spokeThree\cup \rim \subseteq F_3$.  By construction, $\widetilde{\rim} \subseteq F_1$ and $\widetilde{\rim} \subseteq F_3$.  We conclude that 
$\spokeOne\cup \widetilde{\rim} \subseteq F_1$ and 
$\spokeThree \cup \widetilde{\rim} \subseteq F_3$.  Hence, 
$\spokeOne\cup \widetilde{\rim} \in  \scplexC$ and 
$\spokeThree \cup \widetilde{\rim} \in \scplexC$; in other words,
$\widetilde{ \mathcal{F}}$ satisfies~\ref{itm:D1i}.
Finally, 
$\rim \subseteq \widetilde{\rim}$ implies that $\tk_{\code}(\widetilde{\scfaceT})\subseteq \tk_{\code}(\scfaceT)$, and this containment implies that~\ref{itm:D1ii} is inherited from $ \mathcal{F}$ to $\widetilde{ \mathcal{F}}$.
\end{proof}

So far, we have shown that for certain types of wheels, rims can not be codewords, but always bubble up to a max-intersection face.  We conjecture that these results generalize to all wheels. 

\begin{conjecture}\label{conj:bubbleup}
Let $\code$ be a neural code. 
\begin{description}
    \item[\namedlabel{itm:conj-i}{(i)}] If $\WheelFull$ is a wheel of $\code$, then $\rim\not\in \code$.
    \item[\namedlabel{itm:conj-ii}{(ii)}] If $\code$ has a wheel, then $\code$ has a wheel $\WheelFull$ 
    in which 
    $\scfaceT$ is a max-intersection face of $\scplexC$. 
\end{description}
\end{conjecture}
Conjecture~\ref{conj:bubbleup} holds for sprockets and wire wheels (Proposition~\ref{prop:missingrim}, Corollary~\ref{cor:tau-is-max}, and Proposition~\ref{prop:wire-wheel-max-intersection}).
Also, Conjecture~\ref{conj:bubbleup}\ref{itm:conj-ii} -- suitably adjusted to accommodate the fact that wheel frames generate but are not themselves wheels -- holds for wheel frames (Proposition~\ref{prop:wheel-frame-max-intersection}). 

We end this section by describing the intuition behind Conjecture~\ref{conj:bubbleup}\ref{itm:conj-ii}. 
We know that having a wheel 
guarantees non-convexity (Theorem~\ref{thm:wheel}), while
being max-intersection-complete guarantees convexity (Proposition~\ref{prop:summary-prior-results}). 
Hence, the presence of a wheel in a code $\code$ must somehow force a max-intersection face to be excluded from $\code$. Our examples and results suggest that this excluded max-intersection face is always a rim.  

\section{Reducible and decomposable codes} \label{sec:reduce-decompose}
One aim of this article is to use wheels to classify codes on 6 neurons.  
A first step is to exclude those codes that are somehow equivalent to codes on fewer neurons.  
Such codes are described in this section:
reducible codes -- those with irrelevant neurons (Section~\ref{sec:reduced-codes}) -- and decomposable codes -- those 
built from simpler codes (Section~\ref{sec:decomposable-codes}).  
We also prove some preliminary results on decomposable codes on up to 6 neurons (Section~\ref{sec:apply-dec-codes}).

In what follows, we use the following definition:  
let $\code$ be a code on $n$ neurons, and 
$\chi$
a subset of~$[n]$. 
 The \textit{restricted code} 
 $\code |_{\chi}$ is $\{c\cap \chi\ |\ c\in \code \}.$  If
 $\fullreal$ is a realization of $\code$, then 
we let 
$\realiz |_{\chi}$ denote the realization $\{U_i\}_{i\in \chi}$ of $\code |_{\chi}$.  

\begin{lemma} \label{lem:restrict}
Let $\code$ be a code on $n$ neurons, and let $\chi \subseteq [n]$.  If $\code$ is convex, then the restricted code $\code |_{\chi}$ is also convex.
\end{lemma}

\begin{proof}
If $\{U_i\}_{i \in [n]}$ is a convex realization of $\code$,
then $\{U_i\}_{i\in \chi }$ is a convex realization of 
$\code|_{\chi }$.
\end{proof}

\subsection{Reducible codes} \label{sec:reduced-codes}
 
The following terminology, introduced by Jeffs~\cite[\S 3]{jeffs2020morphisms}, captures the case of a superfluous neuron. 


\begin{definition} \label{def:trunk-redundant-reduced}
Let $\code$ be a code on $n$ neurons.
\begin{enumerate}
    \item[(1)] \label{item:trivial}
    A neuron $i$ is \textit{trivial} in $\code$ if $i$ is not in any codeword of $\code$, that is, $\tk_{\code}(\{i\}) = \varnothing$.
    \item[(2)] \label{item:redundant}
    A neuron $i$ is \textit{redundant} in $\code$ if  $i$ is not trivial in $\code$ and there exists $\scfaceS \subseteq [n]$ with $i\not\in \scfaceS$ such that 
        $\tk_{\code}(\{i\}) = \tk_{\code}(\scfaceS)$.
    \item[(3)] \label{item:reduced}
    The code $\code$ is \textit{reduced} if it has no redundant neurons and no trivial neurons. 
\end{enumerate}

We say that a code is \textit{reducible} if it is not reduced.

\end{definition}

The following result 
is essentially due to Jeffs (see \cite[Lemma 3.11 and Corollary 4.2]{jeffs2020morphisms}):

\begin{proposition} \label{prop:reduced}
Let $\code$ be a code on $n$ neurons.  Assume $j$ is a trivial or redundant neuron of $\code$.  Then $\code$ is convex if and only if the restricted code $\code|_{[n] \smallsetminus \{j\} }$ is convex.
\end{proposition}

\begin{proof}
If $j$ is trivial, then $\code = \code|_{[n] \smallsetminus \{j\}}$ and so the result follows. 

Now assume that $j$ is redundant.  
It follows that $\tk_{\code}(\{ j \}) = \tk_{\code}(\scfaceS)$ for some $\scfaceS \subseteq [n]$ with $j \notin \scfaceS$. 
The implication ``$\Rightarrow$'' is Lemma~\ref{lem:restrict}.  

For ``$\Leftarrow$'', assume that the restricted code $\code|_{[n] \smallsetminus \{j\} }$ is convex.  Let $\{ U_i \}_{i \in [n] \smallsetminus \{j\}}$ be a convex realization of the restricted code. 
We claim that $\sigma \neq \emptyset$.  Indeed, if $\sigma = \emptyset$, then 
because $\emptyset \in \code$ (recall Assumption~\ref{assum:empty-codeword}), 
we have that 
$\emptyset \in 
\tk_{\code}(\scfaceS) =
\tk_{\code}(\{ j \}) 
$, which is a contradiction.  So, $\sigma$ is nonempty and hence $U_{\sigma}$ is a convex, open set.  Let $U_j:=U_{\sigma}$, and now it is straightforward to check that $\{U_i\}_{i \in [n]}$ is a convex realization of $\code$. 
\end{proof}


\subsection{Decomposable codes} \label{sec:decomposable-codes}



This subsection considers codes $\code$ that can be ``decomposed'' into
two smaller codes $\code_1$ and $\code_2$ so that a convex realization of $\code$, if it exists, can be obtained by ``embedding'' a convex realization of $\code_1$ into that of $\code_2$ (see Figure~\ref{fig:real-6neur} and Theorem~\ref{thm:decomposable}).  

\begin{definition}
A code $\code$ on $n$ neurons is \textit{decomposable} if there exist disjoint subsets $\varphi, \psi \subsetneq [n]$ with $\varphi \neq \varnothing$ 
such that:
\begin{enumerate}
    \item[(i)] \label{itm:decompose-1}
         $\psi\in \code$, and 
    \item[(ii)] \label{itm:decompose-2}
    every $c\in \code$ that contains at least one neuron of $\varphi$ has the form 
    $c = \widetilde{\varphi} \cup \psi$ for some $\widetilde{\varphi} \subseteq \varphi$.
\end{enumerate}
We call $\code |_{\varphi}$ the \textit{embedded code} and 
$\code |_{[n]\smallsetminus \varphi}$ the \textit{ambient code}.  Also, $\psi$ is the \textit{ambient codeword}.
\end{definition}

%
%
%
%
%
%
%

\begin{example} \label{ex:code-decomposable-again} 
The code  
$\code = \{ {\bf 2356, 123}, 235, 236, 12, 14, 23, 1,2,4,\varnothing\}$
from Example~\ref{ex:code-decomposable} is decomposable, where $\varphi = \{5,6\}$ and $\psi = \{2,3\}$. 
Indeed, the codewords $c\in \code$ intersecting $\varphi$ are 
$235 =  \{5\} \cup \psi$, 
$236 = \{6\} \cup \psi $, and $2356 = \{5,6\} \cup \psi $.
Recall the realization $\realiz$ of $\code$
depicted in Figure~\ref{fig:real-6neur}.
Note that $\realiz$ can be obtained by first drawing a realization of 
$\code |_{[n]\smallsetminus \varphi} = 
\{ {\bf 123, 14}, 
12,  23, 
1, 2, 4, 
\varnothing
\}$ 
and then placing a realization of 
$\code |_{\varphi} = 
\{ {\bf 56},
5,6, 
\varnothing
\}$ inside the 
atom of the ambient codeword 23. 
\end{example}

\subsection{Application to codes on up to 6 neurons} \label{sec:apply-dec-codes}

We saw 
that for the code $\code$
in Example~\ref{ex:code-decomposable-again}, 
a convex realization was obtained by placing a convex realization of the embedded code inside that of the ambient code.  
This was possible because the atom of the ambient codeword has nonempty interior.  Is this always possible?  
The main result of this subsection, Theorem~\ref{thm:decomposable}, 
gives an affirmative answer 
for codes on up to 6 neurons.  We do not know, however, whether this theorem extends to codes on more neurons.

\begin{theorem}\label{thm:decomposable}
 Suppose that $\code$ is a decomposable code on up to 6 neurons
 with embedded code $\code |_{\varphi}$ and ambient code $\code |_{[n]\smallsetminus \varphi}$.
 Then $\code$ is convex if and only if $\code |_{\varphi}$ and $\code |_{[n]\smallsetminus \varphi}$ are convex.
\end{theorem} 

The proof of Theorem~\ref{thm:decomposable}, which appears at the end of this section, uses the idea of nondegeneracy introduced by Cruz \textit{et al.}~\cite[\S 2]{cruz2019open}.  Indeed, 
 condition~\ref{itm:nondegen-1} of the following definition 
 guarantees that, for a decomposable code, the atom of the ambient codeword 
 has full dimension, and hence we can place within it a realization of the embedded code (see Proposition~\ref{prop:decomposable-convex} below).

\begin{definition}\cite[Definition 2.10]{cruz2019open}  \label{def:nondeg} 
For a collection 
$\fullreal$ of subsets of 
$\bbR^d$, 
consider the following properties:
\begin{description}
    \item[\namedlabel{itm:nondegen-1}{(i)}] 
    For all
    $\scfaceS \subseteq [n]$, 
    the set  $(\cap_{i\in c}U_i) \smallsetminus (\cup_{j \in [n] \smallsetminus c}U_j)$ is either empty or 
    \textit{top-dimensional}, i.e., every nonempty intersection with an open subset
    of $\bbR^d$ has nonempty interior. 
    \item[\namedlabel{itm:nondegen-2}{(ii)}] For all nonempty $\scfaceS\subseteq [n]$, we have $\bigcap_{i\in \scfaceS} \partial U_i\subseteq \partial (\bigcap_{i\in \scfaceS}U_i).$
\end{description}
Then $\realiz$ is \textit{top-dimensional} if condition (i) holds, and is \textit{nondegenerate} if 
both~\ref{itm:nondegen-1} and~\ref{itm:nondegen-2} hold.
\end{definition}
We say that a code $\code$ is 
\textit{top-dimensionally convex} (respectively, 
\textit{nondegenerately convex}) if it has a convex realization $\fullreal$ that is top-dimensional (respectively, nondegenerate).  

\begin{proposition} \label{prop:decomposable-convex}
 Suppose that $\code$ is a decomposable code with embedded code $\code |_{\varphi}$ and ambient code $\code |_{[n]\smallsetminus \varphi}$.  
If $\code |_{\varphi}$ is 
top-dimensionally 
convex and $\code |_{[n]\smallsetminus \varphi}$ is convex, then $\code$ is convex.
\end{proposition}

\begin{proof}
Assume that $\code |_{\varphi}$ is 
top-dimensionally 
convex and $\code |_{[n]\smallsetminus \varphi}$ is convex.  Relabel the neurons so that $\varphi = \{1,2,\dots, k\}$ and $[n] \smallsetminus \varphi = \{k+1,k+2, \dots, n \}$.  
Let $\mathcal{V}= \{V_i\}_{i=1}^k$ be a convex realization for $\code |_{\varphi}$ and 
 $\mathcal{W}= \{W_i\}_{i=k+1}^n$ a 
 top-dimensional 
 convex realization for $\code |_{[n] \smallsetminus \varphi}$.  We may assume that both $\mathcal{V}$ and $\mathcal{W}$ are realizations in $\mathbb{R}^d$, for some $d$.  (Indeed, every realization $\mathcal{U}$ in some $\mathbb{R}^{d_1}$ can be elevated to a realization in $\mathbb{R}^{d_1+d_2}$ by taking the product of each $U_i$ in $\mathcal{U}$ with $(0,1)^{d_2}$.)
 
 As $\mathcal{W}$ is 
 top-dimensional 
 and $\psi$ is a codeword of $\code$, there exists an open ball $B$ in $\mathbb{R}^d$ that is strictly contained in the atom for the codeword $\psi$.  Also, recall that we may assume that every $V_i$ in the realization $\mathcal{V}$ is contained in an open ball $\widetilde{B}$ in $\mathbb{R}^d$ (cf.~\cite[Remark 2.19]{new-obs}).  Let $F: \widetilde{B} \to B$ be the scaling bijection between the two balls (which 
 preserves convexity and is a homeomorphism).
 Then, by construction, $\{ F(V_i) \}_{i=1}^k \cup \{ W_i \}_{i=k+1}^n$ is a convex realization of $\code$. 
\end{proof}

In light of Proposition~\ref{prop:decomposable-convex}, we need only show that convexity and 
top-dimensional 
convexity are equivalent for codes on up to 5 neurons (see Proposition~\ref{prop:fiveneur-nondeg} below), 
and then Theorem~\ref{thm:decomposable} will follow.  
To prove Proposition~\ref{prop:fiveneur-nondeg}, we need two results on nondegenerate convexity,
the first of which follows directly from the results of Cruz \textit{et al.}
\cite[Proposition 4.3 and Lemma 3.1]{cruz2019open}:

\begin{lemma} \label{lem:maxint-nondeg}
Every max-intersection-complete code is nondegenerately convex.
\end{lemma}

\begin{proof}
\cite[Proposition 4.3 and Lemma 3.1]{cruz2019open} construct a convex realization for every max-intersection-complete code. It is straightforward to verify that this realization is also nondegenerate.
\end{proof}

The following result is also essentially due to Cruz \textit{et al.}~\cite{cruz2019open}:
\begin{lemma}[Monotonicity of top-dimensional and nondegenerate convexity] 
\label{lem:nondeg-monotone}
Let $\code$ be a neural code, and let 
$\calD$ be a neural code such that $\code \subseteq \calD \subseteq \scplexC$.
\begin{itemize}
    \item[(i)] \label{item:top-dim-monotone}
    If $\code$ is top-dimensionally convex, then $\calD$ is also top-dimensionally convex.     

    \item[(ii)] \label{item:nondeg-monotone}
    If $\code$ is nondegenerately convex, then $\calD$ is also nondegenerately convex.    
\end{itemize}
\end{lemma}

\begin{proof}
Both 
(i) and 
(ii) follow directly from the proof of~\cite[Theorem~1.3]{cruz2019open}, if we can add the following assertion to the statement of~\cite[Lemma~3.1]{cruz2019open}: ``Also, if $\realiz$ is top-dimensional, then $\calV$ can also be chosen to be top-dimensional''.  Indeed, the proof of~\cite[Lemma~3.1]{cruz2019open} directly accommodates this extra assertion once we add the following statement (which is identical to the prior one) to~\cite[Lemma~A.7]{cruz2019open}: ``Also, if $\realiz$ is top-dimensional, then $\calV$ can also be chosen to be top-dimensional.'' 
The proof of this assertion is achieved by starting from the proof of~\cite[Lemma~A.7]{cruz2019open}, replacing the two occurrences of the word ``non-degenerate'' with ``top-dimensional'', and then deleting the final sentence of that proof.
\end{proof}

In light of Lemma~\ref{lem:nondeg-monotone}, to show top-dimensional convexity for convex codes on up to 5 neurons, 
we can proceed by analyzing one simplicial complex $\scplex$ at a time: 
it suffices to check that $\mincode$,  the minimal code of $\scplex$,
is top-dimensionally convex.  
(Recall that 
$\mincode$ 
is the smallest code -- with respect to inclusion -- with neural complex $\scplex$ that has no local obstructions.)

However, there is one simplicial complex for which a different approach must be taken,
namely, the simplicial complex of 
the non-convex code $\mathcal{C} ^{\star}$ (with no local obstructions) from Example~\ref{ex:lienshiu}:
\begin{align} \label{eq:complex-counterexample-code}
    \scplexCstar ~=~ \textrm{the simplicial complex with facets }2345, ~ 123, ~ 134, ~145~.
\end{align}
Note that  
$\mincodestar = \code^{\star}$. 
Part~\ref{item:counterexample-cpx} of the next result clarifies which codes with neural complex~\eqref{eq:complex-counterexample-code} are convex, while 
part~\ref{item:5-neuron-min-code}, 
which is due to Goldrup and Phillipson~\cite{goldrup2020classification}, pertains to all other simplicial complexes on 5 vertices.

\begin{proposition}[Convexity of codes on 5 neurons] \label{prop:five-neur-top-dim} 
Let $\scplex$ be a connected simplicial complex on 5 vertices.
\begin{description}
    \item[\namedlabel{item:5-neuron-min-code}{(i)}]
    If $\scplex$ is \uline{not} isomorphic to the simplicial complex~\eqref{eq:complex-counterexample-code},
    then the code
$\mincode$ is top-dimensionally convex. 
    \item[\namedlabel{item:counterexample-cpx}{(ii)}]
    If $\scplex$ is the simplicial complex~\eqref{eq:complex-counterexample-code},
    then for
    a code $\code$ with neural complex $\scplexC=\scplex$,
    the following are equivalent: 
        \begin{description}
            \item[\namedlabel{item:a-convex}{(a)}]
            $\code$ is convex, 
            \item[\namedlabel{item:b-top-d-convex}{(b)}]
            $\code$ is top-dimensionally convex, and
            \item[\namedlabel{item:c-codewords}{(c)}]
            $\code$ contains at least one of the following codewords: $1$, $234$, and $345$.  
        \end{description}
\end{description}
\end{proposition}

\begin{proof}
\ref{item:5-neuron-min-code}.
If $\mincode$ is max-intersection-complete, the result follows from Lemma~\ref{lem:maxint-nondeg}.
By~\cite[Theorem 3.1]{goldrup2020classification} (in that article, an isomorphic copy of $\mincodestar$ is given the name ``C4''), 
all remaining minimal codes, besides $\mincodestar$, have convex realizations depicted in~\cite[Appendix B]{goldrup2020classification}. 
These realizations are easily seen to be top-dimensional.

\ref{item:counterexample-cpx}. 
The implication \ref{item:b-top-d-convex}$\Rightarrow$\ref{item:a-convex} is by definition.  Next, 
the contrapositive of \ref{item:a-convex}$\Rightarrow$\ref{item:c-codewords} follows directly from Proposition~\ref{prop:original-counterexample-wheel-frame}.  
Now we prove 
\ref{item:c-codewords}$\Rightarrow$\ref{item:b-top-d-convex}.
The code $\code^{\star} \cup \{1 \}$ is max-intersection-complete and thus, by Lemma~\ref{lem:maxint-nondeg}, is nondegenerately convex -- and so, by definition, top-dimensionally convex.  
The code $\code^{\star} \cup \{234 \}$ is also top-dimensionally convex:
such a realization will appear in the forthcoming work of 
Maga\~{n}a and Phillipson~\cite{magana-phillipson}. 
The code $\code^{\star} \cup \{345 \}$ is also top-dimensionally convex: after relabeling the neurons via the permutation $(25)(34)$, the resulting code is the code $\code^{\star} \cup \{234 \}$, which we already analyzed. 
Thus, by Lemma~\ref{lem:nondeg-monotone}(i), every code $\code$
with neural complex $\scplexC=\scplex$ that contains $1$, $234$, or $345$ is top-dimensionally convex. 
\end{proof}

\begin{remark} \label{rmk:min-cvx-codes}
Proposition~\ref{prop:five-neur-top-dim}\ref{item:counterexample-cpx} 
implies that $\code^{\star} \cup \{1 \}$, 
$\code^{\star} \cup \{234 \}$, and 
$\code^{\star} \cup \{345 \}$ 
are the minimal (with respect to inclusion) convex codes with neural complex equal to $\scplexCstar$.  
This corrects an error in~\cite[Remark 3.5]{lienkaemper2017obstructions}, where it was asserted that $\code^{\star} \cup \{234, 345 \}$ is such a minimal convex code.
\end{remark}

We now show the equivalence of convexity and top-dimensional convexity for codes on up to 5 neurons.  

\begin{proposition}\label{prop:fiveneur-nondeg} 
Let $\code$ be a code on $5$ or fewer neurons. 
Then $\code$ is convex if and only if $\code$ is top-dimensionally convex.  
\end{proposition}

\begin{proof}
The direction $\Leftarrow$ holds by definition.  For $\Rightarrow$, let $\code$ be a convex code on $n$ neurons, where $n \leq 5$.  If $n \leq 4$, then 
Proposition~\ref{prop:summary-prior-results}(\ref{itm:third}) 
and Lemma~\ref{lem:maxint-nondeg} together imply that
$\code$ is nondegenerately convex and hence, by definition, top-dimensionally convex.

Now assume $n=5$.  If the neural complex $\scplexC$ is \uline{not} isomorphic to 
$\scplexCstar$, the simplicial complex shown in~\eqref{eq:complex-counterexample-code}, 
then the result follows from 
Lemma~\ref{lem:nondeg-monotone}(i) and 
Proposition~\ref{prop:five-neur-top-dim}\ref{item:5-neuron-min-code}.  
The only remaining case, therefore, is when $\scplexC$ is isomorphic to
$\scplexCstar$.  
By relabeling neurons, if necessary, we may assume that $\scplexC = \scplexCstar$. 
Now the result follows from Proposition~\ref{prop:five-neur-top-dim}\ref{item:counterexample-cpx}.
\end{proof}

We expect that Proposition~\ref{prop:fiveneur-nondeg} does not extend to codes on more neurons. Related questions were posed recently by Chan {\em et al.}~\cite[Question~2.17]{chan-nondegenerate} and 
Jeffs~\cite[Question~10.1]{jeffs2019embedding}.

We now prove two results, stated earlier, pertaining to codes on up to 5 or 6 neurons.

\begin{proof}[Proof of Theorem~\ref{thm:5-neurons-complete}]
$(\Rightarrow)$ This implication follows from
Proposition~\ref{prop:summary-prior-results}\eqref{itm:second} and Theorem~\ref{prop:wire-frame}. 

$(\Leftarrow)$  Let $n$ denote the number of neurons.  For $n \leq 4$, this result follows from Proposition~\ref{prop:summary-prior-results}\eqref{itm:third}.  
Now assume $n=5$.  We consider two cases.  
If $\scplexC$ is isomorphic to $\scplexCstar$, the simplicial complex shown in~\eqref{eq:complex-counterexample-code}, then the theorem follows from 
Propositions~\ref{prop:original-counterexample-wheel-frame} and~\ref{prop:five-neur-top-dim}\ref{item:counterexample-cpx}. 
If $\scplexC$ is {\em not} isomorphic to $\scplexCstar$, 
then the theorem follows 
from 
Lemma~\ref{lem:nondeg-monotone}(i) and Proposition~\ref{prop:five-neur-top-dim}\ref{item:5-neuron-min-code}. 
\end{proof}

\begin{proof}[Proof of Theorem~\ref{thm:decomposable}]
The implication $\Rightarrow$ follows from Lemma~\ref{lem:restrict}, and $\Leftarrow$ is immediate from Propositions~\ref{prop:decomposable-convex} and~\ref{prop:fiveneur-nondeg}. 
%
\end{proof}

 \section{Codes on 6 neurons} \label{sec:6-neuron}
Having shown that wheels and local obstructions completely characterize non-convexity for codes on up to 5 neurons,
we now turn our attention to codes on 6 neurons. 
Due to the large number of such codes, we restrict our analysis to codes with up to 7 maximal codewords.  Additionally, 
like the classification process of~\cite{goldrup2020classification}, 
we focus on minimal codes, $\mincode$, as in~\eqref{eq:min-code}.  
Recall that when some $\mincode$ is convex, then all codes with neural complex equal to $\scplex$ (and no local obstructions) are convex (Proposition~\ref{prop:summary-prior-results}).

Our results, summarized in Table~\ref{tab:fewfacets}, are quite promising.  Approximately 300 codes are found to be non-convex due to a wheel.  
Also, the number of codes with wheels 
increases with the number of maximal codewords.  Wheels are therefore surprisingly efficient in detecting non-convexity.  

\begin{table}[ht]
    \centering
    \begin{tabular}{  l  c  c  c  c  } 
    \hline
    \multicolumn{1}{c}{~} &  \multicolumn{4}{c}{\textit{Number of maximal codewords}} \\
     & Four & Five & Six & Seven \\
    \hline
    Reducible or decomposable & 203 & 480 & 526 & 341\\
    Max-intersection-complete (thus, convex) & 4 & 79 & 399 & 909\\
    Wheel (thus, non-convex) \\ 
    \quad {\em Wheel frame only} & 0 & 1 & 11 & 36\\
    \quad {\em Sprocket only} & 2 & 6 & 14 & 14\\
    \quad {\em Wheel frame and sprocket} & 1 & 29 & 92 & 108\\
    \quad {\em Wire wheel only} & 0 & 0 & 1 & 1\\
    Unknown & 0 & 96 & 535 & 1169 \\
\hline
    Total & 210 & 691 & 1578 & 2578\\
    \hline
    \end{tabular}
 \caption{Classification of 6-neuron codes with 4 to 7 maximal codewords.  More precisely, the table classifies minimal codes $\mincode$ of all (up to isomorphism) connected simplicial complexes on $6$ vertices with 4 to 7 facets.  This table was obtained using Procedure~\ref{procedure}.  Some codes 
 from the ``Unknown'' row
 are listed in Appendix~\ref{sec:appendix-max}.  
 No code containing a wire wheel also contained another type of wheel.
}
     \label{tab:fewfacets}
\end{table}


Our workflow to produce the data in Table~\ref{tab:fewfacets} was as follows.  We first used {\tt nauty}~\cite{nauty} to enumerate all connected simplicial complexes on $6$ vertices (up to isomorphism). Next, we used {\tt SageMath}~\cite{sagemath} to compute the minimal code of each simplicial complex (see~\cite[Algorithm 4.1]{lienkaemper2017obstructions}).  
Finally, we classified 
 6-neuron minimal codes with exactly four, five,  six, or seven maximal codewords,  
  by applying Procedure~\ref{procedure}.  

 \begin{procedure}[Classifying codes on 6 neurons] \label{procedure}
    ~ \\
    {\sc Input:} A code $\code$ on (up to) 6 neurons. \\
    {\sc Output:} ``Reducible or decomposable'', ``Max-intersection-complete'', ``Wheel'', or ``Unknown''.\\
    {\sc Steps:} 
 \begin{enumerate}
     \item Determine whether $\code$ is reducible\footnote{ 
     In our enumerations, when determining whether a code is reducible, we need only check for redundant neurons.  This is because the simplicial complexes we consider are connected, and so the resulting codes lack trivial neurons.} or decomposable.  (We saw in the previous section
     -- specifically, Proposition~\ref{prop:reduced} and Theorem~\ref{thm:decomposable} --
     that, in terms of convexity, such codes are equivalent to codes on fewer neurons.)  If not, proceed to the next step.
     \item Determine whether $\code$ is max-intersection-complete.  (Such codes are convex, by Proposition~\ref{prop:summary-prior-results}.) If not, proceed to the next step.
     \item Determine whether $\code$ has a 
     sprocket, wire wheel, or wheel frame. 
    If so, the convexity status of $\code$ is ``Wheel''. 
     If not, the status is ``Unknown''.
 \end{enumerate}
\end{procedure}

A package containing the {\tt SageMath} scripts we used to apply Procedure~\ref{procedure}, as well as the resulting data on codes (for instance, which have wheels and what those wheels are), is available on {\tt GitHub}~\cite{alex-github}. 
The scripts we used to perform step (3) of Procedure~\ref{procedure} incorporate several results from Section~\ref{sec:search-2} to make checking for a wheel more efficient.  Also, a selection of the codes whose convexity status is unresolved (those in the ``Unknown'' row in Table~\ref{tab:fewfacets}) are listed in Appendix~\ref{sec:appendix-max}.


\begin{remark}[Codes with up to three maximal codewords] \label{rmk:2-or-3-facets}
For codes with at most three maximal codewords -- and any number of neurons -- convexity has been fully characterized: such codes are convex if and only if they have no local obstructions \cite{johnston2020neural}. We therefore do not include codes with three or fewer maximal codewords in Table~\ref{tab:fewfacets}.
\end{remark}

In light of Remark~\ref{rmk:2-or-3-facets}, we draw another conclusion from Table~\ref{tab:fewfacets}: 
{\em A 6-neuron minimal code~$\code$ with up to four maximal codewords is convex if and only if $\code$ has no local obstructions and no wheel obstructions (more precisely, no sprocket, wire wheel, or wheel frame).}
We do not know whether this observation generalizes, either to non-minimal codes or to codes on 7 or more neurons.

Going forward, there are still many codes whose convexity status is unknown (as seen in Table~\ref{tab:fewfacets}).  Additionally, we would like to extend the classification to codes with more than seven maximal codewords, but this task is currently computationally challenging.  Nevertheless, we hope that results like those in Section~\ref{sec:search-2} will make this approach more tractable.

\section{Pure neural codes} \label{sec:pure}
We saw in the previous section that classifying convex codes on six (or more) neurons is challenging, due to the sheer number of such codes.  A more reasonable task, therefore, is to analyze codes with special properties.  One such property is having a neural complex that is pure, as follows.


\begin{definition} \label{def:pure}
A neural code $\code$ is \textit{pure} if its neural complex $\scplexC$ is pure, 
that is, every facet of $\scplexC$ has the same dimension.
\end{definition}
 
Next, we show that for pure codes of low or high dimension, 
being convex is equivalent to being max-intersection-complete, and thus is easy to check.  
For a discussion of the algorithmic aspects of checking whether a code is max-intersection-complete, see~\cite[\S 6]{de2020neural}.
 
\begin{theorem} \label{thm:pure}
Let $\code$ be a code on $n$ neurons.  
If $\code$ is pure of dimension $0$, $1$, $n-2$, or $n-1$, then 
the following are equivalent:
\begin{description}
    \item[\namedlabel{itm:purecond-i}{(i)}] $\code$ is convex,
    \item[\namedlabel{itm:purecond-ii}{(ii)}] $\code$ has no local obstructions, and
    \item[\namedlabel{itm:purecond-iii}{(iii)}] $\code$ is max-intersection-complete.
\end{description}
\end{theorem}

\begin{proof}
The implications 
\ref{itm:purecond-iii} $\Rightarrow$ \ref{itm:purecond-i} $\Rightarrow$ \ref{itm:purecond-ii} 
are contained in Proposition~\ref{prop:summary-prior-results}.

The implication 
\ref{itm:purecond-ii} $\Rightarrow$ \ref{itm:purecond-iii}
holds for dimension~0 (trivially, as by assumption $\varnothing \in \code$), dimension~1 (this follows easily from \cite[Theorem 1.3]{sparse}), and dimension~$n-1$ (by \cite[Lemma 2.5]{what-makes}). 
Now assume that $\code$ is pure of dimension $n-2$ and has no local obstructions.

Let $F_1,F_2,\dots,F_m$ be distinct facets of $\scplexC$ (with $m \geq 2$); 
we must show that 
$\scfaceS:= \bigcap_{j=1}^m F_j$ is a codeword of $\code$.  
As $\code$ has dimension $n-2$, 
there exist distinct $i_1, i_2,\dots, i_m$
such that $F_j = [n]\smallsetminus \{i_j\}$ (for every $j$). 
Thus, $\scfaceS ~=~ [n]\smallsetminus \{i_1,i_2,\dots,i_m\}.$

We now claim that every facet of $\scplexC$ that contains $\scfaceS$ is one of the $F_j$. 
To this end, suppose $F$ is such a facet.
As $\scplexC$ is pure, $F = [n]\smallsetminus \{i\}$ for some $i$. Then $i\not\in \scfaceS$, and so $i\in \{i_1,i_2,\dots, i_m\}$. Therefore, $i = i_j$ (for some $j = 1,\dots,m$) and thus $F = F_j$, proving the claim.

The claim implies that the facets of $\LkCfaceS$ are precisely the
sets
$$F_j\smallsetminus \scfaceS = \{i_1,i_2,\dots,i_{j-1},\hat{i}_j,i_{j+1},\dots,i_m\},\ {\rm for} ~j = 1,\dots,m.$$ 
Thus, $\LkCfaceS$ is the hollow simplex on $m$ vertices (that is, it contains all possible faces on the $m$ vertices except the top face) and so is not contractible (recall that $m \geq 2$).  Thus, as $\code$ has no local obstructions by assumption, we have $\scfaceS \in \code$, and so \ref{itm:purecond-ii} $\Rightarrow$ \ref{itm:purecond-iii} holds.
\end{proof}

The analogous result for pure codes of dimension 2 or $n-3$ (or any dimension in between) does not hold. 
Indeed, we have already seen examples demonstrating this fact. 
For instance, the code 
from Example~\ref{ex:convex-notmaxint} is a convex code on $n=5$ neurons that fails to be max-intersection-complete, but is pure and of dimension $n-3 = 2$.  Another example is the code in Proposition~\ref{prop:3-sparse-counterex}, which is pure of dimension 2, has no local obstructions, and is non-convex.

We end this section by analyzing pure codes on 6 neurons.
For such codes of dimension 0, 1, 4, or 5, 
convexity is well understood (by Theorem~\ref{thm:pure}).  
We therefore focus on pure 6-neuron codes of dimension 2 or 3. 
We applied Procedure~\ref{procedure} to these codes, and our results are shown in Table~\ref{tab:pure-codes}.  
Around half of these codes 
reduce to smaller codes or are classified as convex or non-convex.  Only 6 are found to have wheels, and many codes have unknown convexity status.  
A partial list of these ``unknown'' codes appears in Appendix~\ref{sec:appendix-pure}, 
and the full list
is available at {\tt GitHub}~\cite{alex-github}.

\begin{table}[ht]
    \centering
    \begin{tabular}{  l  c c   } 
    \hline
    $ $ &  $\dim =2$ & $\dim =3$ \\
    \hline
    Reducible or decomposable & 153 & 36 \\
    Max-intersection-complete (convex) &944 & 32  \\
    Wheel (non-convex) & 0 & 6 \\ 
    Unknown &  1004 & 76   \\ 
    \hline
    Total & 2101 & 150\\ 
    \hline
    \end{tabular}
 \caption{Classification of pure 6-neuron codes
of dimension 2 or 3.  More precisely, the table classifies minimal codes $\mincode$ of all (up to isomorphism) connected simplicial complexes on $6$ vertices that are pure of dimension 2 or 3. 
All $6$ codes with a wheel contain both a wheel frame and a sprocket, but not a wire wheel.}
    \label{tab:pure-codes}
\end{table}

\section{Discussion} \label{sec:discussion}
We have now presented a new tool for inferring non-convexity of neural codes, namely, the wheel. 
Notably, wheels and local obstructions together completely characterize convexity for codes on up to $5$ neurons -- and yield a partial classification of codes on up to $6$ neurons. 

This work sheds more light on the question of which codes are convex,
and also raises interesting new questions. Wheels are
different from prior non-convexity criteria in that they are inherently a
property of the realization studied, and not the code. Currently, we
have no way to determine that a code has no wheels, short of proving
the code is convex (as opposed to the situation with local obstructions).
For these
reasons, the foremost open question in this direction is to find a
complete criterion for the existence of wheels in terms of the code, and
not of its realizations. Needless to say, effective criteria, that can
be combinatorially stated and verified, are the
most desirable.



Another question we have posed is whether, in a wheel, the rim is always a non-codeword that bubbles up to a max-intersection face (Conjecture~\ref{conj:bubbleup}).  A related question arising from our examples is whether every code with a wheel ``bubbles down'' to a wheel in which the union of the spokes, $\spokeOne \cup \spokeTwo \cup \spokeThree$, is disjoint from the rim $\rim$. 
An additional future direction is to generalize wire wheels to accommodate links beyond path graphs; one idea here is to use the concept of order-forcing~\cite{jeffs2020order} to guarantee the existence of the required line segment through the three spokes. Answers to these questions would make searching for wheels more efficient and potentially more fruitful.

Other questions concerning wheels arise in connection with other works
in the literature.  
For instance, we know that detecting convexity is an NP-hard problem~\cite{oriented-matriods-and-codes}. Is the same true for detecting wheels? 
We would also like to investigate the relationship between wheels,
decomposable codes, and the partial order on codes introduced by Jeffs in which the convex codes form a down-set~\cite{jeffs2020morphisms}. 

Returning to our partial classification of codes on $6$ neurons, there
are hundreds of codes for which convexity is not currently known.  To
handle so many codes, we need automatic ways of proving a code is
convex -- and also more ways to preclude convexity.  One approach,
mentioned earlier, is to generalize wheels to accommodate more spokes,
inspired by Jeffs's approach on sunflowers~\cite{jeffs2019sunflowers} (Remark~\ref{rmk:sunflower}).  
%
In fact, it might be that such a generalization is needed to resolve our conjectures.  
Specifically, when a wheel bubbles up, perhaps what is obtained is a generalized wheel having more than three spokes.

Finally, we know that wheels preclude a code from being open-convex,
but not necessarily closed-convex. 
A counterpart to wheels for ruling out closed-convexity (independently of local obstructions), via ``rigid structures'', was given recently by Chan {\em et al.}~\cite{chan-nondegenerate}.  On the other hand, order-forcing can be used to construct codes that have no local obstructions and yet are neither open-convex nor closed-convex~\cite{jeffs2020order}.  Going forward, we expect that all of these concepts -- wheels, local obstructions, rigid structures, and order-forcing, plus reducible and decomposable codes -- will aid in future classifications of neural codes, both in terms of open-convexity and closed-convexity.

\subsection*{Acknowledgements}
AR and AS thank Amzi Jeffs, Caitlin Lienkaemper, and Nora Youngs for helpful discussions.
AR also thanks Lienkaemper for help with enumerate simplicial complexes.  
The authors also thank Jeffs and Christian De Los Santos for helpful comments on an earlier draft. 
AS was supported by the NSF (DMS-1752672), and LFM was supported by a
Simons Collaboration Grant for Mathematicians.

\bibliographystyle{siam}
\bibliography{Bibliography.bib}

\begin{thebibliography}{10}

\bibitem{alex-github}
\url{https://github.com/aruysdeperez/WheelsCode.git}.

\bibitem{sprocket-image}
{\em A sprocket and roller chain}.
\newblock \url{https://en.wikipedia.org/wiki/Sprocket#/media/File:Chain.gif}.

\bibitem{chan-nondegenerate}
{\sc P.~Chan, K.~Johnston, J.~Lent, A.~Ruys~de Perez, and A.~Shiu}, {\em
  Nondegenerate neural codes and obstructions to closed-convexity}, Preprint,
  {\tt arXiv:2011.04565},  (2020).

\bibitem{new-obs}
{\sc A.~Chen, F.~Frick, and A.~Shiu}, {\em Neural codes, decidability, and a
  new local obstruction to convexity}, SIAM J.\ Applied Algebra and Geometry, 3
  (2019), pp.~44--66.

\bibitem{cruz2019open}
{\sc J.~Cruz, C.~Giusti, V.~Itskov, and B.~Kronholm}, {\em On open and closed
  convex codes}, Discrete Comput.\ Geom., 61 (2019), pp.~247--270.

\bibitem{what-makes}
{\sc C.~Curto, E.~Gross, J.~Jeffries, K.~Morrison, M.~Omar, Z.~Rosen, A.~Shiu,
  and N.~Youngs}, {\em What makes a neural code convex?}, SIAM J.\ Applied
  Algebra and Geometry, 1 (2017), pp.~222--238.

\bibitem{curto2013neural}
{\sc C.~Curto, V.~Itskov, A.~Veliz-Cuba, and N.~Youngs}, {\em The neural ring:
  an algebraic tool for analyzing the intrinsic structure of neural codes},
  Bull.\ Math.\ Biol., 75 (2013), pp.~1571--1611.

\bibitem{no-go}
{\sc C.~Giusti and V.~Itskov}, {\em A no-go theorem for one-layer feedforward
  networks}, Neural Comput., 26 (2014), pp.~2527--2540.

\bibitem{goldrup2020classification}
{\sc S.~A. Goldrup and K.~Phillipson}, {\em Classification of open and closed
  convex codes on five neurons}, Adv.\ Appl.\ Math., 112 (2020), p.~101948.

\bibitem{jeffs2019sunflowers}
{\sc R.~A. Jeffs}, {\em Sunflowers of convex open sets}, Adv.\ Appl.\ Math.,
  111 (2019), p.~101935.

\bibitem{jeffs2020morphisms}
\leavevmode\vrule height 2pt depth -1.6pt width 23pt, {\em Morphisms of neural
  codes}, SIAM J.\ Applied Algebra and Geometry, 4 (2020), pp.~99--122.

\bibitem{jeffs2019embedding}
{\sc R.~A. Jeffs}, {\em Embedding dimension phenomena in intersection complete
  codes}, Sel. Math. New Ser., 18 (2022).

\bibitem{jeffs2020order}
{\sc R.~A. Jeffs, C.~Lienkaemper, and N.~Youngs}, {\em Order-forcing in neural
  codes}, Preprint, {\tt arXiv:2011.03572},  (2020).

\bibitem{Jeffs2018convex}
{\sc R.~A. Jeffs and I.~Novik}, {\em Convex union representability and convex
  codes}, Int.\ Math.\ Res.\ Notices,  (2019).

\bibitem{sparse}
{\sc R.~A. Jeffs, M.~Omar, N.~Suaysom, A.~Wachtel, and N.~Youngs}, {\em Sparse
  neural codes and convexity}, Involve, a Journal of Mathematics, 12 (2019),
  pp.~737--754.

\bibitem{johnston2020neural}
{\sc K.~Johnston, A.~Shiu, and C.~Spinner}, {\em Neural codes with three
  maximal codewords: Convexity and minimal embedding dimension}, Preprint, {\tt
  arXiv:2008.13192},  (2020).

\bibitem{oriented-matriods-and-codes}
{\sc A.~Kunin, C.~Lienkaemper, and Z.~Rosen}, {\em Oriented matroids and
  combinatorial neural codes}, Preprint, {\tt arXiv:2002.03542},  (2020).

\bibitem{lienkaemper2017obstructions}
{\sc C.~Lienkaemper, A.~Shiu, and Z.~Woodstock}, {\em Obstructions to convexity
  in neural codes}, Adv.\ Appl.\ Math., 85 (2017), pp.~31--59.

\bibitem{magana-phillipson}
{\sc E.~Maga\~{n}a and K.~Phillipson}, {\em Classifying neural codes as closed
  or open convex}, In preparation,  (2021).

\bibitem{nauty}
{\sc B.~D. McKay and A.~Piperno}, {\em Practical graph isomorphism, {II}}, J.\
  Sci.\ Comput., 60 (2014), pp.~94--112.

\bibitem{Oke1}
{\sc J.~O'Keefe and J.~Dostrovsky}, {\em The hippocampus as a spatial map.
  {P}reliminary evidence from unit activity in the freely-moving rat}, Brain
  Res., 34 (1971), pp.~171--175.

\bibitem{de2020neural}
{\sc A.~Ruys~de Perez, L.~F. Matusevich, and A.~Shiu}, {\em Neural codes and
  the factor complex}, Adv.\ Appl.\ Math., 114 (2020), p.~101977.

\bibitem{sagemath}
{\sc {The Sage Developers}}, {\em {S}ageMath, the {S}age {M}athematics
  {S}oftware {S}ystem ({V}ersion 6.7)}, 2015.
\newblock {\tt https://www.sagemath.org}.

\bibitem{williams}
{\sc R.~Williams}, {\em Strongly maximal intersection-complete neural codes on
  grids are convex}, Appl.\ Math.\ Comput., 336 (2018), pp.~162--175.

\end{thebibliography}

%

\appendix
\section{Codes on 6 neurons with 5 to 7 maximal codewords and unknown convexity status} \label{sec:appendix-max}
Recall from Table~\ref{tab:fewfacets} that among the minimal codes (with connected simplicial complex) on 6 neurons, many have unknown convexity status (that is, assessing whether such a code is convex can not be done automatically).
There are 
96 (respectively, 535 or 1169)
such codes with 5 (respectively, 6 or 7) maximal codewords;
we list 10 codes from this set in Section~\ref{sec:list-5-max} (respectively,~\ref{sec:list-6-max} or~\ref{sec:list-7-max}).
The complete list can be found at GitHub~\cite{alex-github}.

\subsection{Codes with 5 maximal codewords} \label{sec:list-5-max}

\begin{enumerate}


    \item 
$\{\textbf{2346}, \textbf{145}, \textbf{456}, \textbf{12}, \textbf{13}, 45, 46, 1, 2, 3, \emptyset\}$


    \item 
$\{\textbf{1456}, \textbf{245}, \textbf{346}, \textbf{12}, \textbf{13}, 45, 46, 1, 2, 3, \emptyset \}$


    \item 
$\{\textbf{245}, \textbf{346}, \textbf{456}, \textbf{12}, \textbf{13}, 45, 46, 1, 2, 3, \emptyset \}$


    \item 
$\{\textbf{2456}, \textbf{135}, \textbf{156}, \textbf{12}, \textbf{34}, 15, 56, 1, 2, 3, 4, \emptyset \}$


    \item 
$\{\textbf{2456}, \textbf{156}, \textbf{235}, \textbf{12}, \textbf{34}, 25, 56, 1, 2, 3, 4, \emptyset \}$


    \item 
$\{\textbf{2356}, \textbf{134}, \textbf{135}, \textbf{146}, \textbf{12}, 13, 14, 35, 1, 2, 6, \emptyset  \}$


    \item 
$\{\textbf{2356}, \textbf{134}, \textbf{135}, \textbf{246}, \textbf{12}, 13, 26, 35, 1, 2, 4, \emptyset \}$


    \item 
$\{\textbf{1456}, \textbf{134}, \textbf{135}, \textbf{246}, \textbf{12}, 13, 14, 15, 46, 1, 2, \emptyset\}$


    \item 
$\{\textbf{3456}, \textbf{134}, \textbf{135}, \textbf{246}, \textbf{12}, 13, 34, 35, 46, 1, 2, 3, \emptyset \}$


    \item 
$\{\textbf{134}, \textbf{135}, \textbf{256}, \textbf{346}, \textbf{12}, 13, 34, 1, 2, 5, 6, \emptyset \}$
\end{enumerate}

\subsection{Codes with 6 maximal codewords} \label{sec:list-6-max}

\begin{enumerate}

    \item 
$\{\textbf{145}, \textbf{246}, \textbf{456}, \textbf{12}, \textbf{13}, \textbf{23}, 45, 46, 1, 2, 3, \emptyset \}$


    \item 
$\{ \textbf{3456}, \textbf{145}, \textbf{246}, \textbf{12}, \textbf{13}, \textbf{23}, 45, 46, 1, 2, 3, \emptyset \}$


    \item 
$\{\textbf{235}, \textbf{256}, \textbf{456}, \textbf{12}, \textbf{13}, \textbf{14}, 25, 56, 1, 2, 3, 4, \emptyset \}$


    \item 
$\{\textbf{2456}, \textbf{156}, \textbf{235}, \textbf{12}, \textbf{13}, \textbf{14}, 25, 56, 1, 2, 3, 4, \emptyset \}$


    \item 
$\{\textbf{2345}, \textbf{2356}, \textbf{156}, \textbf{12}, \textbf{13}, \textbf{14}, 235, 56, 1, 2, 3, 4, \emptyset \}$


    \item 
$\{\textbf{3456}, \textbf{156}, \textbf{235}, \textbf{12}, \textbf{13}, \textbf{24}, 35, 56, 1, 2, 3, 4, \emptyset\}$


    \item 
$\{\textbf{2345}, \textbf{156}, \textbf{256}, \textbf{12}, \textbf{13}, \textbf{14}, 25, 56, 1, 2, 3, 4, \emptyset \}$

 
    \item 
$\{\textbf{235}, \textbf{256}, \textbf{456}, \textbf{12}, \textbf{13}, \textbf{24}, 25, 56, 1, 2, 3, 4, \emptyset \}$ 

    \item 
$\{\textbf{1456}, \textbf{235}, \textbf{256}, \textbf{12}, \textbf{13}, \textbf{24}, 25, 56, 1, 2, 3, 4, \emptyset \}$


    \item 
$\{\textbf{235}, \textbf{356}, \textbf{456}, \textbf{12}, \textbf{13}, \textbf{24}, 35, 56, 1, 2, 3, 4, \emptyset \}$

\end{enumerate}

\subsection{Codes with 7 maximal codewords} \label{sec:list-7-max}


\begin{enumerate}


\item 
$\{\textbf{146}, \textbf{156}, \textbf{246}, \textbf{12}, \textbf{13}, \textbf{23}, \textbf{45}, 16, 46, 1, 2, 3, 4, 5, \emptyset \}$


\item 
$\{\textbf{145}, \textbf{146}, \textbf{245}, \textbf{346}, \textbf{12}, \textbf{13}, \textbf{23}, 14, 45, 46, 1, 2, 3, \emptyset \}$


\item 
$\{\textbf{145}, \textbf{146}, \textbf{156}, \textbf{245}, \textbf{12}, \textbf{13}, \textbf{23}, 14, 15, 16, 45, 1, 2, 3, \emptyset \}$


\item 
$\{\textbf{145}, \textbf{146}, \textbf{245}, \textbf{256}, \textbf{12}, \textbf{13}, \textbf{23}, 14, 25, 45, 1, 2, 3, 6, \emptyset \}$


\item 
$\{\textbf{145}, \textbf{156}, \textbf{245}, \textbf{346}, \textbf{12}, \textbf{13}, \textbf{23}, 15, 45, 1, 2, 3, 4, 6, \emptyset \}$


\item 
$\{\textbf{145}, \textbf{156}, \textbf{246}, \textbf{456}, \textbf{12}, \textbf{13}, \textbf{23}, 15, 45, 46, 56, 1, 2, 3, 5, \emptyset \}$


\item 
$\{\textbf{3456}, \textbf{145}, \textbf{156}, \textbf{246}, \textbf{12}, \textbf{13}, \textbf{23}, 15, 45, 46, 56, 1, 2, 3, 5, \emptyset  \}$


\item 
$\{\textbf{145}, \textbf{245}, \textbf{346}, \textbf{456}, \textbf{12}, \textbf{13}, \textbf{23}, 45, 46, 1, 2, 3, \emptyset \}$


\item 
$\{\textbf{145}, \textbf{246}, \textbf{356}, \textbf{456}, \textbf{12}, \textbf{13}, \textbf{23}, 45, 46, 56, 1, 2, 3, \emptyset \}$


\item 
$\{\textbf{3456}, \textbf{156}, \textbf{245}, \textbf{12}, \textbf{13}, \textbf{14}, \textbf{23}, 45, 56, 1, 2, 3, 4, \emptyset \}$

\end{enumerate}

\section{Pure codes on 6 neurons with unknown convexity status} \label{sec:appendix-pure}
Recall that Table~\ref{tab:pure-codes} classifies minimal $6$-neuron codes for which the neural complex is connected and pure.  Among such codes that are pure of dimension $2$ (respectively, dimension $3$), 
there are 1004 (respectively, 76) with unknown convexity status.  Partial lists of these codes are appear in Sections~\ref{sec:dim-2} and~\ref{sec:dim-3}, and the complete lists are posted to GitHub~\cite{alex-github}.

\subsection{Codes of dimension $2$} \label{sec:dim-2}
Among the 1004 ``unknown'' minimal $6$-neuron codes that are pure of dimension 2, 
the number of maximal codewords is between $5$ and $14$.  
Below we list one such code for each possible number of maximal codewords (from $5$ to $14$). 


\begin{enumerate}


    \item
    $\{\textbf{123}, \textbf{124}, \textbf{126}, \textbf{135}, \textbf{456}, 12, 13, 4, 5, 6, \emptyset \}$

    \item
    $\{\textbf{123}, \textbf{124}, \textbf{135}, \textbf{156}, \textbf{245}, \textbf{246}, 12, 13, 15, 24, 5, 6, \emptyset\}$
    
    \item
    $\{\textbf{123}, \textbf{124}, \textbf{135}, \textbf{145}, \textbf{236}, \textbf{356}, \textbf{456}, 12, 13, 14, 15, 23, 35, 36, 45, 56, 1, 3, 5, \emptyset \}$
    
    \item
    $\{\textbf{123}, \textbf{124}, \textbf{126}, \textbf{145}, \textbf{156}, \textbf{234}, \textbf{356}, \textbf{456}, 12, 14, 15, 16, 23, 24, 45, 56, 1, 2, 3, 5, \emptyset \}$

    \item
    $\{\textbf{123}, \textbf{124}, \textbf{126}, \textbf{134}, \textbf{136}, \textbf{145}, \textbf{234}, \textbf{256}, \textbf{345}, 12, 13, 14, 16, 23, 24, 26, 34, 45, 1, 2, 3, 4, 5, \emptyset \}$
    
    \item
    $\{\textbf{123}, \textbf{124}, \textbf{125}, \textbf{126}, \textbf{134}, \textbf{135}, \textbf{156}, \textbf{245}, \textbf{346}, \textbf{356}, 12, 13, 14, 15, 16, 24, 25, 34, 35, 36, 56, 1,$
    
    $ 2, 3, 5, \emptyset \}$
    
    \item
    $\{\textbf{123}, \textbf{124}, \textbf{125}, \textbf{126}, \textbf{134}, \textbf{135}, \textbf{136}, \textbf{146}, \textbf{156}, \textbf{245}, \textbf{346}, 12, 13, 14, 15, 16, 24, 25, 34, 36, 46,$
    
    $1, 2, 3, 4, 6, \emptyset \}$
    
    \item
    $\{\textbf{123}, \textbf{124}, \textbf{125}, \textbf{135}, \textbf{145}, \textbf{156}, \textbf{235}, \textbf{236}, \textbf{246}, \textbf{256}, \textbf{345}, \textbf{356}, 12, 13, 14, 15, 23, 24, 25, 26, 35,$
    
    $36, 45, 56, 1, 2, 3, 5, 6, \emptyset\}$
    
    \item
    $\{\textbf{123}, \textbf{124}, \textbf{126}, \textbf{134}, \textbf{135}, \textbf{136}, \textbf{145}, \textbf{234}, \textbf{235}, \textbf{245}, \textbf{345}, \textbf{346}, \textbf{356}, 12, 13, 14, 15, 16, 23, 24,
    $
    
    $25, 34, 35, 36, 45, 1, 2, 3, 4, 5, \emptyset\}$
    
    \item
    $\{\textbf{123}, \textbf{124}, \textbf{125}, \textbf{126}, \textbf{135}, \textbf{136}, \textbf{146}, \textbf{156}, \textbf{235}, \textbf{236}, \textbf{245}, \textbf{256}, \textbf{345}, \textbf{356}, 12, 13, 14, 15, 16, 23,$
    
    $24, 25, 26, 35, 36, 45, 56, 1, 2, 3, 5, 6, \emptyset \}$

\end{enumerate}
\subsection{Codes of dimension $3$} \label{sec:dim-3}
\begin{enumerate}


    \item 
$\{\textbf{1234}, \textbf{1235}, \textbf{1246}, \textbf{1256}, \textbf{1345}, 123, 124, 125, 126, 134, 135, 12, 13, \emptyset \}$


    \item 
$\{\textbf{1234}, \textbf{1235}, \textbf{1236}, \textbf{1256}, \textbf{1345}, 123, 125, 126, 134, 135, 12, 13, \emptyset \}$


    \item 
$\{\textbf{1234}, \textbf{1236}, \textbf{1256}, \textbf{1345}, \textbf{2345}, 123, 126, 134, 234, 345, 15, 25, 34, 1, 2, 5, \emptyset \}$


    \item 
$\{\textbf{1234}, \textbf{1236}, \textbf{1256}, \textbf{1345}, \textbf{2346}, 123, 126, 134, 234, 236,
15, 23, 1, \emptyset\}$ 


    \item 
$\{\textbf{1234}, \textbf{1235}, \textbf{1256}, \textbf{1345}, \textbf{2356}, 123, 125, 134, 135, 235, 256, 13, 25, \emptyset \}$

 
    \item
$\{\textbf{1234}, \textbf{1235}, \textbf{1236}, \textbf{1245}, \textbf{1456}, \textbf{2345}, 123, 124, 125, 145, 234, 235, 245, 12, 16, 23, 24, 25, 1, 
$

$2, \emptyset \}$

 
    \item
$\{\textbf{1234}, \textbf{1235}, \textbf{1246}, \textbf{1256}, \textbf{1345}, \textbf{3456}, 123, 124, 125, 126, 134, 135, 345, 12, 13, 46, 56, 4, 5, 6,$

$\emptyset\}$    
 
    \item
$\{\textbf{1234}, \textbf{1235}, \textbf{1246}, \textbf{1256}, \textbf{1345}, \textbf{2345}, 123, 124, 125, 126, 134, 135, 234, 235, 345, 12, 13, 23, $

$34, 35, 3, \emptyset\}$

 
    \item
$\{\textbf{1234}, \textbf{1235}, \textbf{1246}, \textbf{1256}, \textbf{1345}, \textbf{2456}, 123, 124, 125, 126, 134, 135, 246, 256, 12, 13, 26, 45, 4, $

$5,  \emptyset\}$    


    \item
$\{\textbf{1234}, \textbf{1235}, \textbf{1236}, \textbf{1246}, \textbf{1256}, \textbf{1345}, 123, 124, 125, 126, 134, 135, 12, 13, \emptyset\}$

 \end{enumerate}

\end{document}